\documentclass[11pt]{article}
\usepackage{geometry}
\usepackage{macros}
\usepackage{enumitem}
\usepackage{amsthm}
\usepackage{mathtools}
 \usepackage[ruled,vlined]{algorithm2e}
\usepackage{fullpage}
\usepackage{times}
\usepackage[english]{babel}
\usepackage[utf8]{inputenc}
\usepackage{color}
\usepackage[T1]{fontenc}    % use 8-bit T1 fonts
\usepackage{booktabs}       % professional-quality tables
\usepackage{amsfonts}       % blackboard math symbols
\usepackage{nicefrac}       % compact symbols for 1/2, etc.
\usepackage{microtype}      % microtypography

\newcommand{\lbd}{\lambda}
\newcommand{\Lbd}{\Lambda}

\begin{document}
\title{Bregman Augmented Lagrangian Method and Its Acceleration}

\author{
	Shen Yan\thanks{Department of Industrial and Enterprise Systems Engineering (ISE), University of Illinois at Urbana-Champaign (UIUC), Urbana, IL 61801, USA. Emails:
	\href{mailto:sheny2@illinois.edu}{sheny2@illinois.edu}, 
	\href{mailto:niaohe@illinois.edu}{niaohe@illinois.edu}. 
	}
	\and
	Niao He\footnotemark[1]
}
\date{}
\maketitle

%\textcolor{red}{arxiv version.}

\begin{abstract}
We study the Bregman Augmented Lagrangian method (BALM) for solving convex problems with linear constraints. For classical Augmented Lagrangian method, the convergence rate and its relation with the proximal point method is well-understood. However, the convergence rate for BALM has not yet been thoroughly studied in the literature. In this paper, we analyze the convergence rates of BALM in terms of the primal objective  as well as the feasibility violation. We also develop, for the first time, an accelerated Bregman proximal point method, that improves the convergence rate from $\cO(1/\sum_{k=0}^{T-1}\eta_k)$ to $\cO(1/(\sum_{k=0}^{T-1}\sqrt{\eta_k})^2)$, where $\{\eta_k\}_{k=0}^{T-1}$ is the sequence of proximal parameters. When applied to the dual of linearly constrained convex programs, this leads to the construction of an accelerated BALM, that achieves the improved rates for both primal and dual convergences. 
\end{abstract}

\section{Introduction}
In this paper, we re-examine the general class of convex programs with linear constraints, including equality constrained problems:
\begin{equation}\label{eq:linear-equality-problem}
    \min_{x \in \cX} \; f(x)\quad {\rm s.t.} \;A x = b, 
\end{equation}
and inequality constrained problems:
\begin{equation}\label{eq:linear-inequality-problem}
    \min_{x \in \cX} \; f(x)\quad {\rm s.t.} \; A x \leq b,
\end{equation}
where $f(x)$ is a closed and convex function, $\cX\subset\RR^n$ is a closed and convex set, and $A\in\RR^{m\times n}$, $b\in\RR^m$ are given matrix and vector. Such problems occur pervasively in machine learning, signal processing, and many other engineering fields, including basis pursuit, support vector machine, distributed learning, and finding the optimal policy for Markov decision problems.

The classical augmented Lagrangian method (ALM), originally introduced in~\cite{hestenes1969multiplier,powell1969method}, has been one of the most fundamental and popular algorithms for solving problems with linearly constrained convex programs; see e.g,~\cite{bertsekas2014constrained} for a comprehensive overview. Particularly for \eq{eq:linear-equality-problem} and \eq{eq:linear-inequality-problem}, the key steps for ALM are as simple as follows:
\begin{align}
  x_{k+1} &=\begin{cases} \arg \min_{x\in \mathcal{X}}\{f(x) + \lambda_k^\top (Ax-b) + \frac{\eta_k}{2} \|Ax-b\|^2_2\}, & \text{ in case of }\eq{eq:linear-equality-problem} \\
  \arg \min_{x\in \mathcal{X}}\{f(x) + \frac{1}{2\eta_k}\|[\lambda_k + \eta_k(Ax-b)]_+\|^2_2\}, & \text{ in case of } \eq{eq:linear-inequality-problem}
  \end{cases}\\
  \lambda_{k+1} &= \begin{cases}\lambda_{k} + \eta_k (Ax_{k+1}-b), & \text{ in case of } \eq{eq:linear-equality-problem} \\
  [\lambda_{k} + \eta_k (Ax_{k+1}-b)]_+, &\text{ in case of }\eq{eq:linear-inequality-problem}
  \end{cases}
\end{align}
where $\eta_k$ is the proximal parameter at step $k$ and $[x]_+$ stands for taking $\max(x,0)$ for each element.  It is well-known that ALM can be interpreted as applying the proximal point method on the Lagrangian dual~\cite{rockafellar1976augmented}. Defining the Lagrange function as $L(x,\lambda)=f(x)+\lambda^\top (Ax-y)$, the proximal minimization perspective allows us to write the ALM for \eq{eq:linear-equality-problem} and \eq{eq:linear-inequality-problem} in a unified way:
\begin{equation}\label{eq:ALM}
(x_{k+1},\lambda_{k+1})\in\argmax_{\lambda\in\Lambda}\min_{x\in\cX}\left\{ L(x,\lambda)-\frac{1}{2\eta_k}\|\lambda-\lambda_k\|_2^2\right\},
\end{equation}
where $\Lambda = \mathbb{R}^m$ for equality-constrained problem \eq{eq:linear-equality-problem}, and  $\Lambda = \mathbb{R}^m_+$ for inequality-constrained problem \eq{eq:linear-inequality-problem}. Although we focus on linearly constrained problems in this paper, it is worth mentioning that all the results in this paper can be immediately generalized to convex functional constrained problems: $\min_{x \in \cX} \; f(x),\;\; {\rm s.t.} \; G(x) \leq 0$, where $G(x) = (g_1(x), g_2(x), \cdots, g_m(x))$, and $\{g_i(x)\}_{i=1,2,\cdots,m}$ are closed and convex.

The convergence of ALM has been extensively studied in the literature; due to the large volume of literature on this topic, we only list a few results here. The asymptotic convergence in the convex case was provided in \cite{rockafellar1976augmented} from the proximal minimization viewpoint. Understanding the non-asymptotic convergence of ALM and the iteration complexity of  its inexact variations has been the main focus in several recent works. For example, for the linear equality constrained problem \eq{eq:linear-equality-problem}, \cite{lan2016iteration} analyzed the convergence rate for the primal problem, assuming the subproblems are only approximately solved through some first-order subroutines. \cite{xu2017iteration} generalized the results to include both equality and inequality constrained problems \eq{eq:linear-inequality-problem}. There also exist recent works applying ALM to non-convex problems \cite{chen2017augmented, sahin2019inexact}. Moreover, analysis for other variants of ALM also exist, e.g., the linearized augmented Lagrangian method \cite{ouyang2013stochastic, xu2017accelerated}. Besides, ALM and its variants are also widely used in many applications, such as Lasso problems \cite{li2018highly}, sparse principle component analysis \cite{lu2012augmented}, semidefinite programming\cite{yang2015sdpnal}, etc.

As for the acceleration of the ALM, there currently exist two such schemes, which can be derived from applying G\"uler's 1st and 2nd accelerated proximal point methods \cite{guler1992new, kim2019accelerated} to the dual problem, respectively. See \cite{beck2009fast, tseng2010approximation, lin2018catalyst, nedelcu2014computational} and \cite{he2010acceleration, ke2017accelerated, kang2013accelerated} for more details on each scheme. Most of these works only proved an accelerated convergence rate of the dual problem, instead of the primal convergence. For example, \cite{he2010acceleration, ke2017accelerated, kang2013accelerated} applied the G\"uler's 2nd accelerated scheme  to the dual problem, and showed that the Lagrangian residual satisfies $L(x^*,\lambda^*)-L(x_T,\lambda_T)\leq \cO(1/T^2)$, where $(x^*,\lambda^*)$ corresponds to the optimal solution and Lagrange multiplier. Notice that this only implies an accelerated rate for the dual convergence. In fact, generally speaking,  this algorithm could fail to attain the same accelerated $\cO(1/T^2)$ rate in terms primal convergence. We provide one such example in the \hyperref[appendix]{Appendix}\footnotemark.
\footnotetext{We show that for simple linear programs, while the dual objective is proved to converge in the rate of $\cO(1/T^2)$, the primal objective and constraint violation converges in the rate of $\cO(1/T)$.} As far as we know, \cite{nedelcu2014computational} first demonstrated the  $\cO(1/T^2)$ rate of primal convergence for linear equality constrained problem, by applying Nesterov's accelerated dual average method \cite{devolder2014first, nesterov2005smooth} to the augmented Lagrangian dual problem. Indeed, this algorithm is equivalent to using G\"uler's 1st accelerated proximal point method to the dual problem\footnotemark.
\footnotetext{This equivalence was not explicitly mentioned in the original paper, and we prove it in the \hyperref[appendix]{Appendix}.}

 % In summary, the convergence rate for the classical ALM is typically $|f(x_T)-f(x^*)| \sim \|Ax_T-b\|(\|[Ax_T-b]_+\|) \sim O(1/\sum_{k=0}^{T-1}\eta_k)$. 

On the other hand, the formulation \eq{eq:ALM} naturally leads to the generalization of \emph{Bregman Augmented Lagrangian Method} (we refer to BALM for short), where the Euclidean distance is replaced with a general Bregman divergence. This can also be seen as a direct application of the Bregman proximal point algorithm (BPP) to the dual problem, which was originally introduced in~\cite{censor1992proximal,teboulle1992entropic}. Let $h(\lambda):\text{int}(\Lambda)\to\RR$ be a strictly convex, continuously differential function on the interior set $\text{int}(\Lambda)$. The Bregman divergence induced by $h$ is given by $D_h(\lambda,\lambda')=h(\lambda)-h(\lambda')-\langle\nabla h(\lambda'),\lambda-\lambda'\rangle$, which is nonnegative and strictly convex. Thus, the key steps of BALM can be viewed as follows:
\begin{equation}\label{eq:BALM}
(x_{k+1},\lambda_{k+1})\in\argmax_{\lambda\in\Lambda}\min_{x\in\cX}\left\{ L(x,\lambda)-\frac{1}{\eta_k}D_h(\lambda,\lambda_k)\right\}.
\end{equation}
One of the most important advantages of using a Bregman divergence as opposed to the Euclidean distance is that the objective of the subproblems becomes twice-differentiable \cite{teboulle2018simplified}. The use of Bregman divergence also allows more flexibility to exploit the geometry of dual domain $\Lambda$, especially for the linear inequality constrained case. The advantages of BALM have also been observed empirically in practice; see e.g.~\cite{yuan2017bregman} for image segmentation applications.  

However, on the theoretical side, the convergence of BALM has only been studied in few works. The asymptotic convergence is proven in~\cite{eckstein1993nonlinear,chen1993convergence} and  \cite{auslender2006interior} when considering generalized Bregman functions.  To the authors' knowledge, the non-asymptotic convergence rate of BALM is still absent in the literature, especially in terms of the original objective and constraint violation of the primal sequences. Moreover, while accelerated BALM and accelerated proximal point algorithm~\cite{guler1992new} has been established in the Euclidean setting, it remains unclear whether such acceleration schemes can be extended to BALM with general Bregman divergences and whether faster convergence rates can be achieved, especially in terms of the primal convergence.

\paragraph{Our contributions} In this paper, we aim to close these theoretical gaps and make the following key contributions. 
\begin{enumerate}
\item Firstly, we establish the ergodic convergence rate of BALM both in terms of the primal optimality and feasibility violation, when solving the linearly constrained convex problems \eq{eq:linear-equality-problem} and \eq{eq:linear-inequality-problem}. Specifically, we show that $|f(\tilde{x}_T)-f(x^*)|\leq \cO(1/\sum_{k=0}^{T-1}\eta_k)$ and $\|A\tilde{x}_T-b\|, \|[A\tilde{x}_T-b]_+\| \leq \cO(1/\sum_{k=0}^{T-1}\eta_k)$ for \eq{eq:linear-equality-problem} and \eq{eq:linear-inequality-problem}, respectively, where $\{\eta_k\}_{k=0}^{T-1}$ are arbitrary positive proximal parameters. Our proof technique is much simpler than existing ones for classical ALM.
%We also analyze the total iteration complexity of an inexact version of BALM when the objectives are smooth and subproblems are only approximately solved. 
\item Secondly, we develop, for the first time, a general accelerated scheme for Bregman proximal point algorithm (acc-BPP) for convex minimization problems. This acceleration scheme is largely inspired from~G\"uler~\cite{guler1992new} with new treatments and leverages Bregman divergences that satisfy the triangle-scaling property~\cite{hanzely2018accelerated}. Comparing to the BPP~\cite{censor1992proximal,teboulle1992entropic}, acc-BPP improves the convergence rate from $\cO(1/\sum_{k=0}^{T-1}\eta_k)$ to $\cO(1/(\sum_{k=0}^{T-1}\sqrt{\eta_k})^2)$, where $\{\eta_k\}_{k=0}^{T-1}$ is the sequence of proximal parameters. We also derive several simple variants of acc-BPP based on the general accelerated scheme, some of which benefit from a different proof of convergence. 

\item Lastly, we apply the acc-BPP algorithm to the duals of the linearly constrained convex optimization problems \eq{eq:linear-equality-problem} and \eq{eq:linear-inequality-problem}, leading to the accelerated Bregman augmented Lagrangian method (acc-BALM). We demonstrate that acc-BALM achieves a faster convergence rate than BALM. In particular, when choosing constant proximal parameters, the primal convergence rate of acc-BALM is $\cO\left(\ln(T)/T^2\right)$ while the rate of BALM is $\cO\left(1/T\right)$. Moreover, our result generalizes current accelerated ALM schemes, and can be applied to both equality and inequality constrained problems. We believe this is the first primal convergence analysis with acceleration for problems under  inequality constraints.
\end{enumerate}

 \paragraph{Related work} We point out that there exist several other extensions of ALM that also exploit Bregman divergences, which are completely different from BALM. For example, the Bregman ADMM algorithm introduced in~\cite{wang2014bregman} simply replaces the Euclidean distance in the primal update with Bregman divergence while keeping the same update for the dual variable. \cite{zhao2015adaptive} instead directly adds an adaptive Bregman divergence to the linearized augmented Lagrangian function at each iteration as a regularization. Our accelerated Bregman proximal point algorithm (acc-BPP) appears to be related to the accelerated Bregman proximal gradient methods (acc-BPG)~\cite{auslender2006interior,tseng2008accelerated,hanzely2018accelerated,gutman2018unified} designed for minimizing the sum of a (relatively) smooth convex function and a non-differential convex function. The latter can be viewed as an extension of Nesterov's accelerated gradient method~\cite{nesterov1988approach} with Bregman proximal mappings. However, there are notable differences. First of all, the acceleration schemes used in these two algorithms are quite distinct from each other. Secondly, the parameter used in the proximal mappings are different, acc-BPG uses fixed parameters (depending on the relative smoothness) while acc-BPP allows arbitrary parameter values, leading to different characterizations of the convergence rates.

 The rest of the paper is organized as follows. In Section~\ref{sec:prelim}, we provide some preliminaries on BALM and BPP, as well as existing convergence results. In Section~\ref{sec:ergodic convergence}, we establish the first ergodic convergence rate of BALM of the primal sequences. In Section~\ref{sec:acc-BPP}, we develop a generic accelerated scheme for BPP and provide its convergence analysis. In Section~\ref{sec:acc-BPP-variants}, we introduce two simple variants of the acc-BPP algorithm and present a different convergence proof from dual averaging perspective. In Section~\ref{sec:acc-BALM}, we derive the accelerated BALM algorithm and prove its improved primal convergence rate. Finally, in Section~\ref{sec:experiments}, we conduct some numerical experiments to compare the performance of BPP vs. acc-BPP, and BALM vs. acc-BALM, which further validate our theoretical findings.

\section{Preliminaries: augmented Lagrangian and proximal point methods}\label{sec:prelim}
In this section, we review some basics of (Bregman) ALM and (Bregman) proximal point algorithms. The connection of these two algorithms has been well established in the literature; see e.g., \cite{rockafellar1976augmented} and \cite{eckstein1993nonlinear}, respectively for the original ALM and BALM. For simplicity of exposition, we will focus mainly on linearly constrained convex programs, described in \eq{eq:linear-equality-problem} and \eq{eq:linear-inequality-problem}.

The Lagrangian dual of linearly constrained convex programs can be written in the form of 
\begin{equation}\label{eq:dual_problem}
\max_{\lambda\in\Lambda}\; d(\lambda), \text{ where } d(\lambda) = \min_{x\in \mathcal{X}}\;\{ L(x,\lambda):=f(x) + \lambda^\top (Ax-b)\},
\end{equation}
where $\Lambda=\RR^m$ for \eq{eq:linear-equality-problem} or $\Lambda=\RR^m_+$ for \eq{eq:linear-inequality-problem}. It has been shown in \cite{rockafellar1976augmented} and \cite{eckstein1993nonlinear} that the (Bregman) ALM can be viewed as applying the (Bregman) proximal point algorithm to the Lagrangian dual problem \eq{eq:dual_problem}. 

Let $h$ be a proper, continuously differentiable, and strictly convex function on $\Lambda\subseteq\RR^m$. The Bregman divergence induced by function $h$ is given as follows:
\[D_h(\lbd,\tilde{\lbd}) := h(\lbd) - h(\tilde{\lbd}) - \nabla h(\tilde{\lbd})^\top (\lbd-\tilde{\lbd}) \quad \forall \lbd\in \Lbd, \tilde{\lbd}\in \Lbd.\] 
By strict convexity, $D_h(\lbd,\tilde{\lbd}) \geq 0$, and $D_h(\lbd,\tilde{\lbd})=0$ only when $\lbd=\tilde{\lbd}$. For example, when $\Lambda=\RR^m$, the simplest choice of Bregman divergence is $D_h(\lbd,\tilde{\lbd}) = \frac{1}{2}\|\lbd-\tilde{\lbd}\|_2^2$, where $h(\lbd) = \frac{1}{2}\|\lbd\|_2^2$; when $\Lambda=\RR^m_+$, a common choice of Bregman divergence is $D_h(\lbd,\tilde{\lbd}) = \sum_i \left (\lbd^{(i)}\log\lbd^{(i)}-\lbd^{(i)}\log{\tilde{\lbd}^{(i)}} - \lbd^{(i)} + \tilde{\lbd}^{(i)} \right)$, where $h(\lbd) = \sum_i \lbd^{(i)}\log \lbd^{(i)}$. We also list the well-known three-point identity property~\cite{chen1993convergence} of Bregman divergence, which will be heavily used in the analysis: $\forall\lbd_1,\lbd_2,\lbd_3\in\Lambda$,
\begin{equation}\label{eq:threepoint}
D_h(\lbd_1,\lbd_2)+D_h(\lbd_2,\lbd_3)-D_h(\lbd_1,\lbd_3)=\langle\nabla h(\lbd_2)-\nabla h(\lbd_3), \lbd_2-\lbd_1\rangle.
\end{equation}
See \cite{censor1992proximal, chen1993convergence,auslender2006interior} for a more detailed discussion on Bregman divergences.

\DontPrintSemicolon
\begin{algorithm}[t]
\KwIn{$\lambda_0 \in \Lambda, \{\eta_k\}_{k\geq0}$}
%\KwOut{Output}
\nl \For{$k\geq 0$}{
\nl $x_{k+1} \in \arg\min_{x\in \mathcal{X}}\{f(x) + \max_{\lambda\in \Lambda}\{\lambda^\top (Ax-b) - \frac{1}{\eta_k}D_h(\lambda, \lambda_k)\}\}$\;
\nl $\lambda_{k+1} = \arg\max_{\lambda\in \Lambda}\{\lambda^\top (Ax_{k+1}-b)-\frac{1}{\eta_k}D_h(\lambda, \lambda_k)\} $\;
}
    \caption{{\bf Bregman Augmented Lagrangian Method (BALM)} \label{alg:BALM}}
\end{algorithm}

Specifically, the Bregman proximal point (BPP) method~\cite{censor1992proximal, chen1993convergence} for solving the Lagrange dual problem follows the recursion:
\begin{align}
    \lambda_{k+1} := \arg\max_{\lambda \in \Lambda}\left\{d(\lambda) - \frac{1}{\eta_k}D_h(\lambda, \lambda_k)\right\}.\label{eq:BPPM}
\end{align}
The operation defined in \eq{eq:BPPM} is also known as the \emph{Bregman proximal operator}. 
Recall that $d(\lambda) = \min_{x\in \mathcal{X}} \{f(x) + \lambda^\top (Ax-b)\}$ is the Lagrange dual function. Solving \eq{eq:BPPM} reduces to solving the convex-concave saddle point problem 
\begin{align}\label{eq:saddle_point_problem}
\max_{\lambda \in \Lambda}\min_{x\in\cX}\;\left\{\Phi_{\eta_k}(x,\lambda;\lambda_k):=f(x)+\lambda^\top (Ax-b) - \frac{1}{\eta_k}D_h(\lambda, \lambda_k)\right\}.
% =\min_{x\in\cX}\max_{\lambda \in \Lambda}\;\{f(x)+\lambda^\top (Ax-b) - \frac{1}{\eta_k}D_h(\lambda, \lambda_k)\}.
\end{align}
Assuming that both $f$ and $h$ are coercive, based on  convex analysis theory~\cite{rockafellar2015convex,ekeland1999convex}, problem \eq{eq:saddle_point_problem} possesses a saddle point, denoted as $(x_{k+1},\lambda_{k+1})$, such that 
$$\Phi_{\eta_k}(x_{k+1},\lambda;\lambda_k)\leq\Phi_{\eta_k}(x_{k+1},\lambda_{k+1};\lambda_k)\leq\Phi_{\eta_k}(x,\lambda_{k+1};\lambda_k),$$ for any $x\in\cX, \lambda\in\Lbd$. Thus, $\lambda_{k+1}=\argmax_{\lambda\in\Lambda} \Phi_{\eta_k}(x_{k+1},\lambda;\lambda_k)$. The Bregman ALM (BALM), described in \algref{alg:BALM}, can be interpreted as iteratively computing the saddle point of the sequence of subproblems \eq{eq:saddle_point_problem}, or consequently, computing the Bregman proximal operator \eq{eq:BPPM}.

In particular, when $\Lambda=\RR^m$ and the image of the gradient of $h$ satisfies that $\text{Im}(\nabla h)=\RR^m$, the updates in BALM can be simplified as follows~\cite{eckstein1993nonlinear}: 
\begin{align*}
x_{k+1}&\in\argmin_{x\in\cX}\left\{f(x)+\frac{1}{\eta_k}h^*(\nabla h(\lambda_k)+\eta_k(Ax-b)\right\};\\
\lbd_{k+1}&=\nabla h^*(\nabla h(\lbd_k)+\eta_k(Ax_{k+1}-b)),
\end{align*}
where $h^*$ is the convex conjugate of the function $h$. Particularly, when setting $h(\lambda)=\frac{1}{2}\|\lambda\|_2^2$, this leads to the classical ALM; when $h(\lambda) = \sum_{i=i}^m \lambda_i \log(\lambda_i)$ leads to the exponential multiplier method \cite{tseng1993convergence}. For various other examples of BALM, please see~\cite{teboulle1992entropic,iusem1999augmented} and references therein.  Finally, we point out that the convergence analysis of BPP in \eq{eq:BPPM} has been well-studied \cite{censor1992proximal, chen1993convergence}. We list some results below for completeness. 

\begin{lemma}[\cite{chen1993convergence}]\label{lem:convergence_PPA} Let $\{\lambda_k\}_{k\geq0}$ be a sequence generated from \eq{eq:BPPM} with positive parameters $\{\eta_k\}_{k\geq 0}$. Let $\lambda^*\in\Lambda$ be an optimal solution to \eq{eq:dual_problem}. The following holds:
\begin{enumerate}
  \item [(a)] $\eta_k(d(\lbd)-d(\lbd_{k+1}))\leq D_h(\lbd,\lbd_{k})-D_h(\lbd,\lbd_{k+1})-D_h(\lbd_{k+1},\lbd_k)$, $\forall \lbd\in\Lambda$;
  \item[(b)] $d(\lbd_k)$ is non-decreasing;
  \item[(c)] $D_h(\lbd^*,\lbd_k)$ is non-increasing; 
  \item[(d)] $d(\lbd^*)-d(\lbd_T)\leq\frac{D_h(\lbd^*,\; \lbd_{0})}{\sum_{k=0}^{T-1}\eta_k}$; if $\sum_{k=0}^\infty\eta_k=\infty$, then $d(\lbd_k)\to d(\lbd^*)$ as $k\to\infty$. 
\end{enumerate}
\end{lemma}
The result (a) can be obtained directly from the optimality condition of \eq{eq:BPPM} and the three-point identity of the Bregman divergence; (b) follows from (a) by setting $\lbd=\lbd_k$; (c) follows from (a) by setting $\lbd=\lbd^*$; and (d) can be obtained by taking the telescoping sum over (a). Moreover, it can be shown that the sequence $\{\lbd_k\}_{k\geq 0}$ converges to some optimal solution $\lbd^*$. Note that the above results hold true for general convex problems in the form of $\max_{\lbd\in\Lbd}d(\lbd)$. 

Lemma \eq{lem:convergence_PPA} immediately implies the convergence of the dual sequence of BALM for solving the linear constrained convex programs. However, establishing the convergence of the primal sequence, both in terms of the optimality and feasibility, still remains elusive.

\section{Ergodic convergence  of BALM}\label{sec:ergodic convergence}
In this section, we provide the ergodic convergence rate analysis of BALM when subproblems are solved exactly. Throughout, we make the following regularity assumptions:
\begin{assumption}
We assume that
\begin{enumerate}
\item The objective function $f(x)$ is closed, convex, and coercive. The set $X$ is also closed and convex. The primal problem \eq{eq:linear-equality-problem} (or \eq{eq:linear-inequality-problem} resp. for inequality constrained case) and its dual \eq{eq:dual_problem} are solvable, and strong duality holds.
\item The function $h$ is a proper, coercive, continuously differentiable, and strictly convex function on $\Lambda\subseteq\RR^m$, where $\Lambda=\RR^m$ for \eq{eq:linear-equality-problem} or $\Lambda=\RR^m_+$ for \eq{eq:linear-inequality-problem}.
\end{enumerate} 
\end{assumption}{}
Denote $(x^*,\lambda^*)\in\cX\times\Lambda$ as a pair of optimal solution and Lagrange multiplier that satisfies the underlying KKT condition.  Thus, $(x^*,\lambda^*)$ is also a saddle point to the Lagrange function $L(x,\lambda)$ on $\cX\times\Lbd$, and satisfies that 
   $ L(x, \lambda^*) - L(x^*, \lambda) \geq 0, \forall x\in \mathcal{X}, \lambda \in \Lambda. $
Note that since $\Lambda = \mathbb{R}^{m}$ or  $\Lambda=\RR^m_+$, one can choose $\lambda=0$, and the above inequality implies that 
\begin{align}
    f(x) - f(x^*) + {\lambda^*}^\top (Ax-b) \geq 0, \quad \forall x\in \mathcal{X}. \label{eq:VIreduced}
\end{align}
This inequality will be used later in the convergence analysis. The next lemma characterizes the one-step behavior of the algorithm,
\begin{lemma} \label{lma:1step}
We have for any $x\in \mathcal{X},\lambda \in \Lambda$,
\begin{align}
    &L(x_{k+1}, \lambda) - L(x, \lambda_{k+1}) 
    % =&f(x_{k+1}) - f(x) + \lambda^\top (Ax_{k+1}-b) - \lambda_{k+1}^\top (Ax-b) \nonumber\\
    \leq \frac{1}{\eta_k}(\nabla h(\lambda_{k+1}) - \nabla h(y_k))^\top (\lambda - \lambda_{k+1}).
\end{align}
Moreover, 
\begin{align}
 L(x_{k+1}, \lambda) - L(x, \lambda_{k+1}) \leq \frac{1}{\eta_k}( D_h(\lambda, \lambda_k) - D_h(\lambda, \lambda_{k+1})).   
\end{align}
\end{lemma}
\begin{proof}
As discussed earlier in Section~\ref{sec:prelim}, BALM is equivalent to solving a convex-concave saddle point problem \eq{eq:saddle_point_problem} each step, thus we have the following optimality conditions 
\begin{align}
  &(x-x_{k+1})^\top (g_{k+1} + A^\top \lambda_{k+1}) \geq 0, \quad \forall x \in \mathcal{X}, \; g_{k+1}\in \partial f(x_{k+1})  \label{eq:1stepOPT_a} \\
  &(\lambda_{k+1}-\lambda)^\top \left(Ax_{k+1}-b-\frac{1}{\eta_k}(\nabla h(\lambda_{k+1}) - \nabla h(\lambda_{k}))\right)\geq 0, \quad \forall \lambda \in \Lambda.\label{eq:1stepOPT_b}
\end{align}
Therefore, we have 
\begin{align*}
    &L(x_{k+1}, \lambda) - L(x, \lambda_{k+1}) \\
    = &f(x_{k+1}) - f(x) + \lambda^\top (Ax_{k+1}-b) + {\lambda_{k+1}}^\top (b-Ax)\\
    % = &f(x_{k+1}) - f(x) + {x_{k+1}}^\top A^\top \lambda_{k+1} - x^\top A^\top \lambda_{k+1} - {x_{k+1}}^\top A^\top \lambda_{k+1} + \lambda^\top Ax_{k+1} -\lambda^\top b +{\lambda_{k+1}}^\top b\\
    = &f(x_{k+1}) - f(x) + (x_{k+1}-x)^\top A^\top \lambda_{k+1} - (\lambda_{k+1}-\lambda)^\top (Ax_{k+1}-b).
  \end{align*}
  Invoking the convexity of $f(x)$ and~\eq{eq:1stepOPT_a}, it follows that
  \begin{align*}
    f(x_{k+1}) - f(x) +(x_{k+1}-x)^\top A^\top \lambda_{k+1}  &\leq (x_{k+1}-x)^\top  g_{k+1} + (x_{k+1}-x)^\top A^\top \lambda_{k+1} \\
    & = (x_{k+1}-x)^\top ( g_{k+1}+ A^\top \lambda_{k+1}) \leq 0.
  \end{align*}
  For the second part, we have
  \begin{align*}
    - (\lambda_{k+1}-\lambda)^\top (Ax_{k+1}-b) &\leq \frac{1}{\eta_k}(\lambda - \lambda_{k+1})^\top (\nabla h(\lambda_{k+1}) - \nabla h(\lambda_{k}))\\
    &= \frac{1}{\eta_k}( D_h(\lambda, \lambda_k) - D_h(\lambda, \lambda_{k+1}) - D_h(\lambda_{k+1},\lambda_{k}))\\
    &\leq \frac{1}{\eta_k}( D_h(\lambda, \lambda_k) - D_h(\lambda, \lambda_{k+1})).
  \end{align*}
  where the first inequality uses \eq{eq:1stepOPT_b}, and the second equality uses the three-point identity of Bregman divergence. Summing up the above two inequalities leads to the desired result.
\end{proof}

Denote the candidate solution
$$\tilde{x}_T = \frac{{\sum}_{k=0}^{T-1}\eta_{k} x_{k+1}}{\sum_{k=0}^{T-1}\eta_k},\; \tilde{\lambda}_T = \frac{\sum_{k=0}^{T-1}\eta_{k} \lambda_{k+1}}{{\sum}_{k=0}^{T-1}\eta_k}.$$
From Lemma \ref{lma:1step}, we can immediately obtain the following result. 
\begin{lemma}\label{lma:PDgap} We have
\begin{align}\label{eq:Lagrange_residual}
    L(\tilde{x}_T, \lambda) - L(x, \tilde{\lambda}_T) \leq \frac{D_h(\lambda, \lambda_0)}{\sum_{k=0}^{T-1}\eta_k}, \quad \forall x\in \mathcal{X}, \lambda \in \Lambda.
\end{align}
Moreover, by setting $x=x^*$, we further have 
 \begin{align}
    f(\tilde{x}_T) - f(x^*) + \lambda^\top (A\tilde{x}_T - b) \leq \frac{D_h(\lambda, \lambda_0)}{\sum_{k=0}^{T-1}\eta_k}, \quad \forall \lambda \in \Lambda. \label{eq:PgapIneq}
\end{align}
\end{lemma}
\begin{proof} To obtain \eq{eq:Lagrange_residual}, we see that
  \begin{align*}
    L(\tilde{x}_T, \lambda) - L(x, \tilde{\lambda}_T)
    \leq &\frac{1}{\sum_{k=0}^{T-1}\eta_k} \sum_{k=0}^{T-1}\eta_k\left( L(x_{k+1}, \lambda) - L(x, \lambda_{k+1})\right)\\
    \leq &\frac{1}{\sum_{k=0}^{T-1}\eta_k} \sum_{k=0}^{T-1}( D_h(\lambda, \lambda_k) - D_h(\lambda, \lambda_{k+1})) \\
    = & \frac{1}{\sum_{k=0}^{T-1}\eta_k}\left(D_h(\lambda, \lambda_0)-D_h(\lambda, \lambda_T)\right)
  \end{align*}
  The first step uses the fact that $L(x,\lambda)$ is convex in $x$ and linear in $\lambda$, and the second step uses Lemma~\ref{lma:1step}. Combining the fact that $D_h(\lambda,\lambda_T)\geq0$, we end up with \eq{eq:Lagrange_residual}. Setting $x=x^*$, the result in \eq{eq:PgapIneq} follows based on the fact that $\tilde{\lambda}_T^\top (Ax^*-b) \leq 0$. 
\end{proof}

% If $\Lambda$ is a compact set, the previous lemma actually gives the convergence rate of the algorithm converging to a saddle point
% \begin{corollary}\label{cor:saddlepoint}
% If $\Lambda$ is bounded, then the updates \eqref{eq:BALM}
% \begin{align*}
%     \max_{\lambda \in \Lambda}L(\tilde{x}_T, \lambda) - \min_{x \in \mathcal{X}}L(x, \tilde{\lambda}_T) \leq \frac{\max_{\lambda \in \Lambda}D_h(\lambda, \lambda_0)}{\sum_{k=0}^{T-1}\eta_k}
% \end{align*}
% \end{corollary}
% However, for linearly-constrained problems, $\Lambda = \mathbb{R}^{m}(\mathbb{R}^{m}_+)$ and is typically unbounded. To directly obtain the convergence rate on the primal problem, we focus on inequality-constrained problems (equality-constrained problems are similar), thus $\Lambda = \mathbb{R}^{m}_+$. For such problems, we notice from \ref{lma:PDgap}, we have
% \begin{align}
%     f(\tilde{x}_T) - f(x^*) + \lambda^\top (A\tilde{x}_T - b) \leq \frac{D_h(\lambda, \lambda_0)}{\sum_{k=0}^{T-1}\eta_k}, \quad \forall \lambda \in \Lambda \label{eq:PgapIneq}
% \end{align}
% To see this, just note that $\tilde{\lambda}_T^\top (Ax^*-b) \leq 0$. This relation will be useful for obtaining primal convergence rate.

The following theorem describes the primal convergence rate of BALM applied to the linearly constrained problems \eq{eq:linear-equality-problem} an \eq{eq:linear-inequality-problem}, both in terms of the optimality and the constraint violation.
\begin{theorem}\label{thm:primalexact}
 Define $\rho_*=2\|\lambda^*\|+1$. \algref{alg:BALM} satisfies that 
\begin{enumerate}
  \item[(a)] For the equality constrained problem \eq{eq:linear-equality-problem}, 
  \begin{align}\label{eq:bound_equality_case}
\max\left\{ |f(\tilde{x}_T) - f(x^*)|, \|A\tilde{x}_{T}-b\|_{\star}\right\}
\leq \frac{\max_{\lambda \in \mathcal{B}_{\rho_*}}D_h(\lambda, \lambda_0)}{\sum_{k=0}^{T-1}\eta_k},
\end{align}
where $\mathcal{B}_{\rho} = \{\lambda\in\RR^m : \|\lambda\|\leq \rho\}$ and $\|\cdot\|_\star$ is the dual norm of $\|\cdot\|$.
  \item[(b)] For the inequality constrained problem \eq{eq:linear-inequality-problem}, 
  \begin{align}\label{eq:bound_inequality_case}
\max\left\{ |f(\tilde{x}_T) - f(x^*)|, \|[A\tilde{x}_{T}-b]_+\|_{\star}\right\}
\leq \frac{\max_{\lambda \in \mathcal{B}^+_{\rho_*}}D_h(\lambda, \lambda_0)}{\sum_{k=0}^{T-1}\eta_k},
\end{align}
where $\mathcal{B}^+_{\rho} = \{\lambda\in\RR^m :\lambda \geq 0, \|\lambda\|\leq \rho\}$ and $[x]_+=\max(x,0)$.
\end{enumerate}
\end{theorem}

\begin{proof} We only focus on the proof for the inequality constrained case. The equality constrained case follows similarly.  Setting $\lambda=0$ in \eq{eq:PgapIneq} implies 
$$f(\tilde{x}_T) - f(x^*) \leq \frac{D_h(0, \lambda_0)}{\sum_{k=0}^{T-1}\eta_k}.$$ Taking maximum over $\lambda\in\mathcal{B}_{\rho}^+$ leads to 
\begin{align}
   f(\tilde{x}_T) - f(x^*) + \rho \|[A\tilde{x}_{T}-b]_+\|_{\star} \leq \frac{\max_{\lambda \in \mathcal{B}^+_{\rho}}D_h(\lambda, \lambda_0)}{\sum_{k=0}^{T-1}\eta_k}, \forall \rho>0. \label{ineq:rho}
\end{align}
Plugging in $x=\tilde{x}_T$ into the equation \eq{eq:VIreduced},  we have
\begin{align}
    f(x^*) -f(\tilde{x}_T) - \|\lambda^*\| \|[A\tilde{x}_{T}-b]_+\|_{\star} \leq f(x^*) -f(\tilde{x}_T) - {\lambda^*}^\top (A\tilde{x}_T - b) \leq 0. \label{ineq:norm}
\end{align}
Now summing together \eq{ineq:rho} and \eq{ineq:norm}, we obtain
\begin{align}
    (\rho - \|\lambda^*\|) \|[A\tilde{x}_{T}-b]_+\|_{\star} \leq \frac{\max_{\lambda \in \mathcal{B}^+_{\rho}}D_h(\lambda, \lambda_0)}{\sum_{k=0}^{T-1}\eta_k}.\label{ineq:rhonorm}
\end{align}
Setting $\rho = 2\|\lambda^*\|+1$ in \eq{ineq:rhonorm} and combining with the fact that $f(\tilde{x}_T) - f(x^*)  \geq -\|\lambda^*\|\cdot \|[A\tilde{x}_{T}-b]_+\|_{\star}$ leads to the desired result in  \eq{eq:bound_inequality_case}. 
\end{proof}

The above result generalizes the existing ergodic convergence result for classical ALM, e.g., in~\cite{xu2017iteration}, which can be viewed as a special case when the Bregman divergence is set to the Euclidean distance.  From the analysis, we see that the convergence of the primal objective and constraint violation heavily depends on the chosen norm used to measure the constraint violation. Note that the rate of primal convergence is essentially in the same order as that of the dual convergence discussed in previous section. When the proximal parameters $\{\eta_k\}_{k\geq0}$ are fixed to a constant, this implies the $\cO(1/T)$ convergence rate of both primal and dual sequences.

\section{A generic acceleration scheme of Bregman Proximal Point method}\label{sec:acc-BPP}
In the seminal work~\cite{guler1992new}, G{\"u}ler proposed the first accelerations of the proximal point algorithm based on Nesterov's acceleration scheme~\cite{nesterov1988approach}. Inexact versions of the accelerated PPA have been later studied in~\cite{he2012accelerated,salzo2012inexact} and  recent work~\cite{lin2018catalyst}. While it seems rather natural to extend the accelerated PPA to the non-Euclidean setting, there exists only few attempts in this direction~\cite{hamdi2011convergence,hamdi2012convergence}. It came to our attention that these existing works contain fatal flaws, both algorithmically and theoretically.

Motivated by \cite{auslender2006interior,guler1992new}, we propose the first theoretically-sound acceleration scheme for Bregman proximal point method, which will later applied to accelerate BALM. Without loss of generality, we consider solving the convex problem 
\begin{equation}\label{eq:objective_PPA}
\max_{\lambda\in\Lambda}\; d(\lambda), 
\end{equation}
where $d(\lambda)$ and $\Lambda$ are closed and convex. The objective $d(\lambda)$ does not have to be the Lagrange dual of the linearly constrained convex program. Let $D_h(\lambda,\lambda'):\Lambda\times\Lambda\to\infty$ be a Bregman divergence induced by some function $h$ that is continuously differentiable and strictly convex on $\Lambda$. In addition, we assume that the Bregman divergence satisfies the so-called~\emph{triangle scaling property}, which turns out to be a crucial assumption to achieve faster rates. The triangle scaling property was introduced recently in~\cite{gutman2018unified,hanzely2018accelerated} for analyzing the convergence of Bregman proximal gradient methods for relatively smooth objective functions. To be specific,

\begin{assumption}\label{ass:TSP} There exists some constant $G>0$ such that the Bregman divergence $D_h$ has the triangle scaling property: for all $\lambda,\lambda_1,\lambda_2\in \text{int}(\Lambda)$,
\begin{align}
    D_h((1-\theta)\lambda + \theta \lambda_1, (1-\theta)\lambda + \theta \lambda_2) \leq G \theta^2 D_h(\lambda_1,\lambda_2),\forall \theta\in[0,1]. \label{eq:trianglescaling}
\end{align}
\end{assumption}

For detailed discussions about this property, see~\cite{hanzely2018accelerated}. For ease of exposition, here we simply adopt $G$ as a uniform constant, which is closely related to the Hessian of the Bregman function. In particular, if the Bregman divergence is both $L_h$-Lipschitz smooth and $\sigma_h$-strongly convex, then $\frac{\sigma_h}{2}\|x-y\|^2\leq D_h(x,y) \leq \frac{L_h}{2}\|x-y\|^2$, thus Assumption~\ref{ass:TSP} is satisfied with $G = L_h/\sigma_h$. 

The general idea for constructing the acceleration scheme is to first define the following sequence of functions recursively:
\begin{align}
\left\{ \begin{array} { l }
  {\phi_0(\lambda) = d(\lambda_0) - AD_h(\lambda, \lambda_0)} \\
  {\phi_{k+1}(\lambda) = (1-\theta_k)\phi_k(\lambda) + \theta_k(d(Jy_k) + \frac{1}{\eta_k}(\nabla h(Jy_k) - \nabla h(y_k))^\top (\lambda - Jy_k))},\\
     \end{array} \right. \label{eq:defphi}
\end{align}
where $y_k$ is any point (to be specified later) and $Jy_k := \arg\max_{\lambda\in \Lambda}\left\{d(\lambda) - \frac{D_h(\lambda, y_k)}{\eta_k}\right\}$. These functions satisfy the following relation,
\begin{lemma} For any $k$ and $\lambda\in\Lambda$, it holds that 
\begin{align}
    d(\lambda) - \phi_{k+1}(\lambda) \leq (1-\theta_k)(d(\lambda) - \phi_k(\lambda)).
\end{align}
\end{lemma}
\begin{proof}
By concavity and optimality condition from the definition of $Jy_k$, we have
\begin{align}
    d(\lambda) \leq d(Jy_k) + \frac{1}{\eta_k}(\nabla h(Jy_k) - \nabla h(y_k))^\top (\lambda - Jy_k). \label{eq:BPPoptimality}
\end{align}
Hence, it immediately implies that
\begin{align}
    d(\lambda) - \phi_{k+1}(\lambda) &= (1-\theta_k)(d(\lambda) - \phi_k(\lambda)) + \theta_k(d(\lambda) - d(Jy_k)\\
    &\qquad - \frac{1}{\eta_k}(\nabla h(Jy_k) - \nabla h(y_k))^\top (\lambda - Jy_k)) \nonumber\\
    &\leq (1-\theta_k)(d(\lambda) - \phi_k(\lambda)).
\end{align}
\end{proof}

Our goal is to obtain $\lambda_k$ such that $d(\lambda_k) \geq \max_{\lambda\in \Lambda} \phi_k(\lambda)$. From the construction of $\phi_k(\lambda)$, we can see that 
\begin{equation}\label{eq:form_phi}
\phi_k(\lambda) = l_k(\lambda) - A_k D_h(\lambda, \lambda_0), 
\end{equation}
where $l_k(\lambda)$ is an affine function, and $A_k = \prod_{i=0}^{k-1}(1-\theta_k)A$. Using the three-point identity of the Bregman divergence, it can be easily shown that
\begin{align}
    \phi_k(\lambda) = \phi_k(\lambda') + \nabla \phi_k(\lambda')^\top (\lambda-\lambda') - A_kD_h(\lambda,\lambda'), \forall \lambda,\lambda' \in \Lambda.
\end{align}
This means that if we let $v_k: = \arg\max_{\lambda\in \Lambda} \phi_k(\lambda)$, we have
\begin{align}
    \phi_k(\lambda) \leq \phi_k(v_k) - A_k D_h(\lambda, v_k), \forall \lambda \in \Lambda. \label{eq:stronglyconvexD}
\end{align}
The following lemma shows how to construct the desired $\lambda_{k+1}$, given that we already have $d(\lambda_k) \geq \phi_k(v_k)$.

\begin{lemma}
Suppose we already have $\lambda_k$ such that $\phi_k(v_k) \leq d(\lambda_k)$, then choosing 
$$\frac{G\theta_k^2}{\eta_k} = (1-\theta_k)A_k, y_k = \theta_k v_k + (1-\theta_k)\lambda_k, \lambda_{k+1}=Jy_k$$ would ensure $\phi_k(v_{k+1}) \leq d(\lambda_{k+1})$.
\end{lemma}\label{lma:induction}
\begin{proof} Denote $\Delta_k=\frac{1}{\eta_k}[\nabla h(Jy_k) - \nabla h(y_k)]$. We can show that
\begin{align*}
    \phi_{k+1}(v_{k+1})\nonumber
     &= \max_{\lambda \in \Lambda} \{(1-\theta_k)\phi_k(\lambda)+ \theta_k(d(Jy_k) + \Delta_k^\top (\lambda - Jy_k))\} \nonumber\\
    &\leq \max_{\lambda \in \Lambda} \{(1-\theta_k)(\phi_k(v_k) - A_k D_h(\lambda, v_k))+ \theta_k(d(Jy_k) + \Delta_k^\top (\lambda - Jy_k))\} \nonumber\\
    &\leq \max_{\lambda \in \Lambda} \{(1-\theta_k)(d(\lambda_k) - A_k D_h(\lambda, v_k))+ \theta_k(d(Jy_k) + \Delta_k^\top (\lambda - Jy_k))\} \nonumber\\
    &\leq d(Jy_k) + \max_{\lambda \in \Lambda} \{-(1-\theta_k)A_k D_h(\lambda, v_k)+ \Delta_k^\top (\theta_k\lambda + (1-\theta_k)\lambda_k - Jy_k)\} \nonumber\\
    &\leq d(Jy_k) + \max_{\lambda \in \Lambda} \{-(1-\theta_k)A_k D_h(\lambda, v_k)+ \frac{1}{\eta_k}D_h(\theta_k\lambda + (1-\theta_k)\lambda_k, y_k)\}. \nonumber
    %  &(\text{choose $y_k = \theta_k v_k + (1-\theta_k)\lambda_k$}) \nonumber\\
    % &\leq d(Jy_k) + \max_{\lambda \in \Lambda} \{-(1-\theta_k)A_k D_h(\lambda, v_k)+ \frac{1}{\eta_k}D_h(\theta_k\lambda + (1-\theta_k)\lambda_k, \theta_k v_k + (1-\theta_k)\lambda_k)\}, 
\end{align*}
Here the first inequality uses \eq{eq:stronglyconvexD}; the second inequality uses the induction hypothesis; the third inequality applies \eq{eq:BPPoptimality} with $\lambda = \lambda_k$; and the last inequality uses the three-point identity \eq{eq:threepoint}. Next, based on assumption \eq{ass:TSP} and the fact that $y_k = \theta_k v_k + (1-\theta_k)\lambda_k$, we can further obtain:
\begin{align}
    \phi_{k+1}(v_{k+1}) &\leq d(Jy_k) + \max_{\lambda \in \Lambda} \Big\{-(1-\theta_k)A_k D_h(\lambda, v_k)\\
    &\qquad\qquad + \frac{1}{\eta_k}D_h(\theta_k\lambda + (1-\theta_k)\lambda_k, \theta_k v_k + (1-\theta_k)\lambda_k)\Big\} \nonumber \\
    &\leq d(Jy_k) + \max_{\lambda \in \Lambda} \{-(1-\theta_k)A_k D_h(\lambda, v_k)+ \frac{G \theta_k^2}{\eta_k}D_h(\lambda, v_k)\} \nonumber
\end{align}
If we choose $\lambda_{k+1}=Jy_k$, and $\frac{G\theta_k^2}{\eta_k} = (1-\theta_k)A_k$, by induction, we have $d(\lambda_{k+1}) \geq \phi_{k+1}(v_{k+1})$.
\end{proof}

Now we are in the position to present the generic accelerated scheme of  Bregman proximal point algorithm (acc-BPP) in \algref{Accelerated BPP1}. The computation of $v_k$ can be carried out in a closed form in most cases, since $\phi_k(\lambda)$ composes an affine term and a Bregman divergence term as expressed in~\eq{eq:form_phi}.
\begin{algorithm}[ht!]
\KwIn{$\lambda_0 \in \Lambda, v_0=\lambda_0, A_0=A \in (0, +\infty), \phi_0(\lbd) = d(\lambda_0) - AD_h(\lambda,\lambda_0), \{\eta_k\}_{k\geq0}, G$}
%\KwOut{Output}
\nl \For{$k\geq 0$}{
\nl Choose $\theta_k$ such that $\frac{ \eta_k A_k(1-\theta_k)}{G} = \theta_k^2$, i.e. $\theta_k = \frac{\sqrt{(A_k\eta_k/G)^2+4A_k\eta_k/G} - A_k\eta_k/G}{2}$\;
\nl $y_k = \theta_k v_k + (1-\theta_k)\lambda_k$\;
\nl $\lambda_{k+1} \in \arg\max_{\lbd\in \Lambda}\{d(\lbd) - \frac{1}{\eta_k}D_h(\lbd, y_k)\}$\;
\nl $A_{k+1} = (1-\theta_k)A_k$\;
\nl $\phi_{k+1}(\lbd) = (1-\theta_k)\phi_k(\lbd) + \theta_k(d(\lambda_{k+1}) + \frac{1}{\eta_k}(\nabla h(\lambda_{k+1}) - \nabla h(y_k))^\top (\lambda-\lambda_{k+1}))$\;
\nl $v_{k+1} = \arg\max_{\lambda\in \Lambda}\phi_{k+1}(\lambda)$
}
    \caption{{\bf Accelerated Bregman Proximal Point Algorithm (acc-BPP)} \label{Accelerated BPP1}}
\end{algorithm}

The following theorem characterizes the convergence rate of \algref{Accelerated BPP1}.

\begin{theorem}
Under Assumption~\ref{ass:TSP}, \algref{Accelerated BPP1} satisfies that
\begin{align}
d(\lambda^*) - d(\lambda_T) \leq \left [\prod_{k=0}^{T-1}(1-\theta_i)\right ] \left(d(\lambda^*) - d(\lambda_0) + AD_h(\lambda^*, \lambda_0)\right), \label{eq:ABPPrate}
\end{align}
and
\begin{align}
    \frac{1}{\left(1+\sqrt{A/G}\sum_{k=0}^{T-1}\sqrt{\eta_k}\right)^2}\leq \prod_{k=0}^{T-1}(1-\theta_i) \leq \frac{1}{\left(1+(\sqrt{A/G}/2)\sum_{k=0}^{T-1}\sqrt{\eta_k}\right)^2}. \label{eq:ABPPrateseq}
\end{align}
\end{theorem}
\begin{proof}
By the construction of $\{\phi_k(\lbd)\}_k$, we have
\[d(\lambda) - \phi_{k}(\lambda) \leq \prod_{i=0}^{k-1}(1-\theta_i) (d(\lambda) - \phi_0(\lambda)),\quad \forall \lambda\in \Lambda.\]
And the inductive construction of $\lambda_k$ guarantees that $d(\lambda_k) \geq \max_{\lambda \in \Lambda} \phi_k(\lambda)$, which proves \eq{eq:ABPPrate}. The inequality \eq{eq:ABPPrateseq} is proved in \cite{guler1992new}.
\end{proof}

The above theorem indicates that acc-BPP improves the convergence rate of BPP from $\cO\left({1}/{\sum_{j=0}^{T-1}\eta_j}\right)$ to $\cO\big({1}/{(\sum_{j=0}^{T-1}\sqrt{\eta_j})^2}\big)$. This recovers the result in~\cite{guler1992new} as a special case.  In particular, if we choose $\eta_k = \eta, k=0,1,\cdots,T-1$ to be a constant, this  automatically leads to the $\cO(1/T^2)$ convergence rate. Note however, $\eta_k$ can be arbitrarily chosen in practice. 

Note that the triangle scaling property of Bregman divergences naturally arises when designing the generic acceleration scheme. While this condition may not be satisfied by some Bregman divergences, it should be noticed from the proof that the condition only needs to hold true at $\theta_k, k=0,1,\cdots$. When $\theta_k$'s are bounded away from $0$, which is almost always the case in practice, there exists a constant $G$ such that the property is satisfied, albeit possibly being large and difficult to estimate. In the numerical experiments, we find that setting $G$ to be any positive constant provides accelerated performance.

\paragraph{Remark} Recall that existing accelerated Bregman proximal gradient methods (ABPG)~\cite{hanzely2018accelerated,gutman2018unified} attain the $\cO(1/T^2)$ rate of convergence when solving the composite optimization, i.e., $\max_{\lambda\in\Lambda} g(\lambda)+d(\lambda)$, where $g(\lambda)$ is (relatively) Lipschitz smooth and $d(\lambda)$ is a simple concave function admitting easy-to-compute Bregman operators.  One may be tempted to think that ABPG reveals an ``accelerated" version of Bregman proximal point method when setting $g(\lambda)=0$. Take the Algorithm 1 in \cite{hanzely2018accelerated} for an example. Setting $g(\lambda)=0$ leads to the following ``accelerated" algorithm
\begin{align}
\left\{ \begin{array} { l }
  {y_k = (1-\theta_k)\mu_k + \theta_k \lbd_k} \\
  {\lbd_{k+1} = \arg\max_{\lbd\in \Lambda}\{d(\lbd) - \theta_k LD_h(\lbd,\lbd_k)\}}\\
  {\mu_{k+1} = (1-\theta_k)\mu_k + \theta_k \lbd_{k+1}}\\
  {\frac{1-\theta_{k+1}}{\theta_{k+1}^2} = \frac{1}{\theta_k^2}}
     \end{array} \right.
\end{align}
The choice of $L$ is arbitrary, but it is fixed once chosen. Based on the convergence analysis in \cite{gutman2018unified}, the above algorithm inherits the convergence rate of $d(\lambda^*) - d(\mu_T) \leq \frac{4LD_h(\lambda^*,\; \mu_0)}{T^2}$. However, it is also shown in~\cite{gutman2018unified} that $\theta_k \leq \frac{2}{k+1}$, or equivalently, the proximal parameters $\eta_k \geq \frac{k+1}{2L}$ in the proximal point scheme. Notably,  if we choose $\eta_k \geq \frac{k+1}{2L}$ in the vanilla BPP,  the achievable convergence rate  would be $\cO\left(\frac{1}{\sum_{k=0}^{T-1}\eta_k}\right) = \cO(1/T^2)$, which already attains the same rate as the above ``accelerated" algorithm. In contrast, the proposed acc-BPP would achieve the rate $\cO\left(\frac{1}{(\sum_{k=0}^{T-1}\sqrt{\eta_k})^2}\right) = \cO(1/T^3)$ with such proximal parameters, which is much faster than $\cO(1/T^2)$ rate. Therefore, we emphasize that the freedom in choosing arbitrary $\{\eta_k\}_{k\geq 0}$ is crucial here, and also distinct our acceleration with the one above. 

% Thus, a true ABPP should be faster than $O(\frac{1}{\sum_{k=0}^{T-1}\eta_k})$, as opposed to just achieving $O(1/t^2)$ by changing step-sizes.

%  As we show below, current ABPG schemes possess major flaws when being considered as a accelerated Bregman proximal point method. Recall that it is proved in \cite{chen1993convergence} the convergence rate for the vanilla BPP
% \begin{align}
%     \lambda_{k+1} = \arg \max_{\lambda\in \Lambda}\{d(\lambda) - \frac{D_h(\lambda, \lambda_k)}{\eta_k}\} \label{eq:BPP}
% \end{align}
% satisfies $d(\lambda^*) - d(\lambda_T) \leq \frac{D_h(\lambda^*, \lambda_0)}{\sum_{k=0}^{T-1}\eta_k}$, where $d(\lambda)$ is a concave function. 

% For an accelerated BPP, we want to still maintain the freedom to choose arbitrary $\{\eta_k\}_{k\geq0}$, because as we have shown in the previous section, the choice of $\{\eta_k\}_{k\geq0}$ affects the convergence rate of the inner problems. If an acceleration scheme requires $\{\eta_k\}_{k\geq0}$ to increase with $k$, then the inner problem becomes harder and harder to solve. 

\section{Two variations of the accelerated Bregman Proximal Point method}\label{sec:acc-BPP-variants}
In this section, we introduce two variations (or special cases) of acc-BPP that enjoy much more compact forms as well as simpler convergence analysis. 

\subsection{Memoryless form}
In the previous generic acceleration scheme, $v_k$ is defined as $v_k := \arg\max_{\lambda\in \Lambda} \phi_k(\lambda)$. This requires keeping track of the explicit form of functions $\{\phi_k\}_{k\geq 0}$ and computing its maximizer, which may not necessarily admit a closed form. This issue can be alleviated by setting $v_k := \arg\max_{\lambda} \phi_k(\lambda)$, instead. In fact, it can be easily seen that the proof still remains valid by doing so, if additional relaxing Assumption~\ref{ass:TSP} to hold on the entire domain of $h$ instead of $\Lambda$. In this case, we can obtain a closed-form for $v_k$. 
\begin{lemma}[\cite{auslender2006interior}]
Let $\phi_k(\lambda)$ be recursively defined as \eq{eq:defphi} with $Jy_k = \lambda_{k+1}$, then $v_k = \arg\max_{\lambda} \phi_k(\lambda)$ satisfies the following recursion relation
\begin{align}
    v_{k+1} &= \arg\max_{\lambda}\left\{\frac{\theta_k}{\eta_k}(\nabla h(\lambda_{k+1}) - \nabla h(y_k))^\top \lambda - A_{k+1}D_h(\lambda, v_k) \right\},\nonumber \\
    &= \arg\max_{\lambda}\left\{\frac{1}{G\theta_k}(\nabla h(\lambda_{k+1}) - \nabla h(y_k))^\top \lambda - D_h(\lambda, v_k) \right\}.
\end{align}
\end{lemma}
This implies an explicit $v$-update:
$v_{k+1} = \nabla h^* \big(\nabla h(v_k) + \frac{1}{G\theta_k} (\nabla h(\lambda_{k+1}) - \nabla h(y_k))\big).$

\begin{algorithm}[ht!]
\KwIn{$\lambda_0 \in \Lambda, v_0=\lambda_0, \theta_0 = 1 , \{\eta_k\}_{k\geq0}, G$}
%\KwOut{Output}
\nl \For{$k\geq 0$}{
\nl $y_k = \theta_k v_k + (1-\theta_k)\lambda_k$\;
\nl $\lambda_{k+1} \in \arg\max_{\lbd\in \Lambda}\{d(\lbd) - \frac{1}{\eta_k}D_h(\lbd, y_k)\}$\;
\nl $v_{k+1} = \nabla h^* \left(\nabla h(v_k) + \frac{1}{G\theta_k} (\nabla h(\lambda_{k+1}) - \nabla h(y_k))\right)$\;
\nl Update $\theta_{k+1}$ such that $\frac{\eta_k}{\theta_{k}^2} = \frac{\eta_{k+1}}{\theta_{k+1}^2} - \frac{\eta_{k+1}}{\theta_{k+1}}$.
}
    \caption{{\bf  acc-BPP2} \label{Accelerated BPP2}}
\end{algorithm}

As a result, acc-BPP simply reduces to \algref{Accelerated BPP2}, which no longer requires to store the representation of $\{\phi_{k}\}$ an has much cheaper memory cost.  The proof for the convergence rate of acc-BPP2 follows exactly that of acc-BPP. Note that choosing $\theta_0 = 1$\footnote{Notice that $\theta_0$ can also be any real number in $(0,1]$, accordingly $A = \frac{G\theta_0^2}{\eta_0(1-\theta_0)}$.} amounts to setting $A=+\infty$ in acc-BPP, thus acc-BPP2 satisfies
\begin{align}
    d(\lambda^*) - d(\lambda_T) \leq \lim_{A\to +\infty}\frac{d(\lambda^*) - d(\lambda_0) + AD_h(\lambda^*, \lambda_0)}{(1+(\sqrt{A/G}/2)\sum_{j=0}^{T-1}\sqrt{\eta_j})^2} = \frac{4GD_h(\lambda^*, \lambda_0)}{\left(\sum_{k=0}^{T-1}\sqrt{\eta_k}\right)^2}
\end{align}

\subsection{Dual averaging form} Below, we show that the above special case of acc-BPP admits another form that resembles the Nesterov's accelerated dual average method \cite{nesterov2005smooth, devolder2014first, tseng2008accelerated}. Recall the definition of $\{\phi_k\}_{k\geq 0}$ and $\{A_k\}_{k\geq 0}$,  the computation of $v_{k+1}$ is also equivalent to 
\begin{align}
    v_{k+1} &= \arg\max_{\lambda\in\Lambda}\Big\{{\sum}_{j=0}^k \frac{\theta_j}{A_{j+1}\eta_j}\left(\nabla h(\lambda_{j+1}) - \nabla h(y_j)\right)^\top \lambda - D_h(\lambda, \lambda_0) \Big\},\nonumber \\ 
    &= \arg\max_{\lambda\in\Lambda}\Big\{{\sum}_{j=0}^k \frac{1}{G\theta_j}\left(\nabla h(\lambda_{j+1}) - \nabla h(y_j)\right)^\top \lambda - D_h(\lambda, \lambda_0) \Big\}.
\end{align}
Hence, acc-BPP can also be rewritten as \algref{Accelerated BPP3}. Based on the dual averaging interpretation, we show that \algref{Accelerated BPP3} admits a simpler convergence proof, which will be further used to prove the primal convergence of accelerated BALM in the next section.

\begin{algorithm}[ht!]
\KwIn{$\lambda_0 \in \Lambda, v_0=\lambda_0, G, \theta_0 =1 , \{\eta_k\}_{k\geq0}$}
%\KwOut{Output}
\nl \For{$k\geq 0$}{
\nl $y_k = \theta_k v_k + (1-\theta_k)\lambda_k$\;
\nl $\lambda_{k+1} \in \arg\max_{\lbd\in \Lambda}\{d(\lbd) - \frac{1}{\eta_k}D_h(\lbd, y_k)\}$\;
\nl $v_{k+1} = \arg\max_{\lambda\in\Lambda}\left\{- GD_h(\lambda, \lambda_0) + \sum_{j=0}^k \frac{\eta_j}{\theta_j}\left(d(\lambda_{j+1}) + \frac{1}{\eta_j}(\nabla h(\lambda_{j+1}) - \nabla h(y_j))^\top (\lambda-\lambda_{j+1})\right)  \right\}$\;
\nl Update $\theta_{k+1}$ such that $\sum_{j=0}^{k+1}\frac{\eta_j}{\theta_j} = \frac{\eta_{k+1}}{\theta_{k+1}^2}$, satisfied by $\frac{\eta_k}{\theta_{k}^2} = \frac{\eta_{k+1}}{\theta_{k+1}^2} - \frac{\eta_{k+1}}{\theta_{k+1}}$.
% \nl Update $\theta_{k+1}$ such that $\frac{\eta_k}{\theta_{k}^2} = \frac{\eta_{k+1}}{\theta_{k+1}^2} - \frac{\eta_{k+1}}{\theta_{k+1}}$.
}
    \caption{{\bf  acc-BPP3} \label{Accelerated BPP3}}
\end{algorithm}

\begin{theorem}\label{thm:accdual}
Let $S_k = \sum_{j=0}^{k}\frac{\eta_j}{\theta_j}$. Under Assumption~\ref{ass:TSP},  \algref{Accelerated BPP3} satisfies the relation: 
% \begin{align}
%     S_k d(\lambda_{k+1}) \geq \max_{\lambda }\left\{-G D_h(\lambda, \lambda_0) + \sum_{j=0}^k \frac{\eta_j}{\theta_j}\left(d(\lambda_{j+1}) + \frac{1}{\eta_j}(\nabla h(\lambda_{j+1}) - \nabla h(y_j))^\top (\lambda-\lambda_{j+1})\right)\right\}
% \end{align}
\begin{align}
    S_k d(\lambda_{k+1}) \geq \max_{\lambda\in\Lambda }\tilde{\phi}_k(\lambda),
\end{align}
where $\tilde{\phi}_k(\lambda):=-GD_h(\lambda, \lambda_0)+{\sum}_{j=0}^k \frac{\eta_j}{\theta_j}\big(d(\lambda_{j+1}) + \frac{1}{\eta_j}(\nabla h(\lambda_{j+1}) - \nabla h(y_j))^\top (\lambda-\lambda_{j+1})\big).$
\end{theorem}
\begin{proof}
We prove the claim by induction. When $k=0, \forall \lambda\in\Lambda, $ 
\begin{align*}
    \tilde{\phi}_0(\lambda)=&-G D_h(\lambda, \lambda_0) + \frac{\eta_0}{\theta_0}d(\lambda_{1}) + \frac{1}{\theta_0} (\nabla h(\lambda_{1}) - \nabla h(y_0))^\top (\lambda-\lambda_{1})\\
    =& -G D_h(\lambda, \lambda_0) + \frac{\eta_0}{\theta_0}d(\lambda_{1}) + \frac{1}{\theta_0}(D_h(\lambda, y_0) - D_h(\lambda, \lambda_1) - D_h(\lambda_1, y_0))\\
    \leq& -G D_h(\lambda, \lambda_0) + \frac{\eta_0}{\theta_0}d(\lambda_{1}) + \frac{1}{\theta_0} D_h(\lambda, \lambda_0)\\
    \leq& \frac{\eta_0}{\theta_0}d(\lambda_{1}).
\end{align*}
Denote $\Delta_k=\frac{1}{\eta_{k}}[\nabla h(\lambda_{k+1}) - \nabla h(y_k)]$. Suppose now the relation is satisfied for $k$, namely,  $\tilde{\phi}_k(\lambda)\leq S_k d(\lambda_{k+1}), \forall \lambda\in\Lambda$. Then we have, 
\begin{align*}
    % &-G D_h(\lambda, \lambda_0) + \sum_{j=0}^{k+1} \frac{\eta_j}{\theta_j}\left(d(\lambda_{j+1}) + \frac{1}{\eta_j}(\nabla h(\lambda_{j+1}) - \nabla h(y_j))^\top (\lambda-\lambda_{j+1})\right) \nonumber \\
    \tilde{\phi}_{k+1}(\lambda)
    \overset{\langle 1\rangle}{=}&\tilde{\phi}_{k}(\lambda)
    % \overset{\langle 1\rangle}{=}& -G D_h(\lambda, \lambda_0) + \sum_{j=0}^{k} \frac{\eta_j}{\theta_j}\left(d(\lambda_{j+1}) 
    % + \frac{1}{\eta_j}(\nabla h(\lambda_{j+1}) - \nabla h(y_j))^\top (\lambda-\lambda_{j+1})\right) \nonumber \\
    % &\quad \quad \quad  
    + \frac{\eta_{k+1}}{\theta_{k+1}}\left(d(\lambda_{k+2}) + \Delta_{k+1}^\top (\lambda-\lambda_{k+2})\right) \nonumber\\
    % \overset{\langle 2\rangle}{\leq} &-G D_h(v_{k+1}, \lambda_0) + \sum_{j=0}^{k} \frac{\eta_j}{\theta_j}\left(d(\lambda_{j+1}) + \frac{1}{\eta_j}(\nabla h(\lambda_{j+1}) - \nabla h(y_j))^\top (v_{k+1}-\lambda_{j+1})\right) - GD_h(\lambda, v_{k+1}) \nonumber\\
    \overset{\langle 2\rangle}{\leq}&\tilde{\phi}_k(v_{k+1})- GD_h(\lambda, v_{k+1}) + \frac{\eta_{k+1}}{\theta_{k+1}}\left(d(\lambda_{k+2}) + \Delta_{k+1}^\top (\lambda-\lambda_{k+2})\right) \nonumber\\
    \overset{\langle 3\rangle}{\leq} & S_k d(\lambda_{k+1}) + \frac{\eta_{k+1}}{\theta_{k+1}}\left(d(\lambda_{k+2}) + \Delta_{k+1}^\top (\lambda-\lambda_{k+2})\right) - GD_h(\lambda, v_{k+1}) \nonumber\\
    \overset{\langle 4\rangle}{\leq} & S_k\left(d(\lambda_{k+2}) + \Delta_{k+1}^\top (\lambda_{k+1}-\lambda_{k+2})\right)   + \frac{\eta_{k+1}}{\theta_{k+1}}\left(d(\lambda_{k+2}) + \Delta_{k+1}^\top (\lambda-\lambda_{k+2})\right) - GD_h(\lambda, v_{k+1})
    % =&S_{k+1}d(\lambda_{k+2})-GD_h(\lambda,v_{k+1})\\
    % &\quad \quad \quad +{S_k}\Delta_{k+1}^\top (\lambda_{k+1}-\lambda_{k+2}) + \frac{\eta_{k+1}}{\theta_{k+1}}\Delta_{k+1}^\top (\lambda-\lambda_{k+2})
\end{align*}
where step $\langle 1\rangle$ is a trivial identity based on the definition of $\tilde{\phi}_{k+1}$, step $\langle 2\rangle$ uses the strong convexity of $\tilde{\phi}_k$, step $\langle 3\rangle$ uses the induction hypothesis, step $\langle 4\rangle$ uses \eq{eq:BPPoptimality}. Now since 
\[\frac{S_k}{\eta_{k+1}(1-\theta_{k+1})} = \frac{\eta_{k+1}(1-\theta_{k+1})/\theta_{k+1}^2 }{\eta_{k+1}(1-\theta_{k+1})} = \frac{1}{\theta_{k+1}^2},\]
we have
\begin{align*}
     &{S_k}\Delta_{k+1}^\top (\lambda_{k+1}-\lambda_{k+2}) + \frac{\eta_{k+1}}{\theta_{k+1}}\Delta_{k+1}^\top (\lambda-\lambda_{k+2})\\
     =& \frac{1}{\theta_{k+1}^2}(\nabla h(\lambda_{k+2}) - \nabla h(y_{k+1}))^\top ((1-\theta_{k+1})\lambda_{k+1} + \theta_{k+1}\lambda - \lambda_{k+2})\\
     \overset{\langle 5\rangle}{\leq}&\frac{1}{\theta_{k+1}^2}D_h((1-\theta_{k+1})\lambda_{k+1} + \theta_{k+1}\lambda, y_{k+1})\\
     \overset{\langle 6\rangle}{\leq}& GD_h(\lambda, v_{k+1}).
\end{align*}
Here step $\langle 5\rangle$ applies the three-point identity \eq{eq:threepoint}, and step $\langle 6\rangle$ uses the triangle-scaling property \eq{eq:trianglescaling}. As a result, it follows that  $\tilde{\phi}_{k+1} (\lambda)\leq S_{k+1} d(\lambda_{k+2}),\forall \lambda\in\Lambda$.
% \begin{align*}
%     &-G D_h(\lambda, \lambda_0) + \sum_{j=0}^{k+1} \frac{\eta_j}{\theta_j}\left(d(\lambda_{j+1}) + \frac{1}{\eta_j}(\nabla h(\lambda_{j+1}) - \nabla h(y_j))^\top (\lambda-\lambda_{j+1})\right)\nonumber\\
%     \overset{\langle 5\rangle}{\leq} & S_{k+1} d(\lambda_{k+2}) + \frac{1}{\theta_{k+1}^2}(\nabla h(\lambda_{k+2}) - \nabla h(y_{k+1}))^\top ((1-\theta_{k+1})\lambda_{k+1} + \theta_{k+1}\lambda - \lambda_{k+2}) -  GD_h(\lambda, v_{k+1})\nonumber\\
%     \overset{\langle 6\rangle}{\leq} & S_{k+1} d(\lambda_{k+2}) + \frac{1}{\theta_{k+1}^2}D_h((1-\theta_{k+1})\lambda_{k+1} + \theta_{k+1}\lambda, y_{k+1}) - GD_h(\lambda, v_{k+1}) \nonumber\\
%     \overset{\langle 7\rangle}{\leq} &  A_{k+1} d(\lambda_{k+2}) + GD_h(\lambda, v_{k+1}) - GD_h(\lambda, v_{k+1}) \leq S_{k+1} d(\lambda_{k+2})
% \end{align*}
% where step $\langle 6\rangle$ uses the three-point identity \cref{eq:threepoint}, and step $\langle 7\rangle$ uses the triangle-scaling assumption \cref{eq:trianglescaling}.
\end{proof}

Recall the optimality condition~\eq{eq:BPPoptimality} for the $\lambda$-update, we have $\forall \lambda \in \Lambda$,
\begin{align*}
    \sum_{j=0}^k \frac{\eta_j}{\theta_j}\left(d(\lambda_{j+1}) + \frac{1}{\eta_j}(\nabla h(\lambda_{j+1}) - \nabla h(y_j))^\top (\lambda-\lambda_{j+1})\right) \geq S_k d(\lambda).
\end{align*}
Combining with the above theorem, this implies that $S_k d(\lambda_{k+1}) \geq -GD_h(\lambda, \lambda_0) + S_k d(\lambda)$. Therefore, we immediately  obtain the convergence result:
\begin{corollary}\label{cor:dualconvergence}
For any $\lambda \in \Lambda$, 
\begin{align}
    d(\lambda) - d(\lambda_{k+1}) \leq \frac{GD_h(\lambda, \lambda_0)}{S_k} = \frac{\theta_k^2GD_h(\lambda, \lambda_0)}{\eta_k}.
\end{align}
\end{corollary}
% \begin{proof}
% By \cref{eq:BPPoptimality}, $\forall \lambda \in \Lambda$,
% \begin{align*}
%     \sum_{j=0}^k \frac{\eta_j}{\theta_j}\left(d(\lambda_{j+1}) + \frac{1}{\eta_j}(\nabla h(\lambda_{j+1}) - \nabla h(y_j))^\top (\lambda-\lambda_{j+1})\right) \geq S_k d(\lambda).
% \end{align*}
% Thus,
% $
%     S_k d(\lambda_{k+1}) \geq -GD_h(\lambda, \lambda_0) + S_k d(\lambda),
% $
% and the claim is proved.
% \end{proof}

Next, we establish a bound on $\theta_k$.
\begin{proposition}\label{prop:thetabound}
\begin{align}
    \frac{\sqrt{\eta_k}}{\sum_{i=0}^{k}\sqrt{\eta_i}} \leq \theta_k \leq \frac{2\sqrt{\eta_k}}{\sum_{i=0}^{k}\sqrt{\eta_i}} \label{eq:thetabound}
\end{align}
\end{proposition}
\begin{proof}
Let $t_i = \frac{1}{\theta_i}$, then the update $\frac{\eta_i}{\theta_{i}^2} = \frac{\eta_{i+1}}{\theta_{i+1}^2} - \frac{\eta_{i+1}}{\theta_{i+1}}$ is equivalent to
\begin{align*}
    t_{i+1}^2 - t_{i+1} - \frac{\eta_i}{\eta_{i+1}}t_i^2 = 0, \text{ i.e., } t_{i+1} = \frac{1+\sqrt{1+4\frac{\eta_i}{\eta_{i+1}}t_i^2}}{2}.
\end{align*}
Therefore,
\begin{align*}
    &\frac{1}{2} + \frac{1}{2}\left(1+2\sqrt{\frac{\eta_i}{\eta_{i+1}}}t_i \right)\geq t_{i+1} \geq \frac{1}{2} + \sqrt{\frac{\eta_i}{\eta_{i+1}}}t_i.
\end{align*}
which further implies 
    $\sqrt{\eta_{i+1}} \geq \sqrt{\eta_{i+1}}t_{i+1} - \sqrt{\eta_i}t_i \geq \frac{1}{2}\sqrt{\eta_{i+1}}$. Taking summation over $i=0,\ldots, k-1$, this leads to
    $$\sum_{i=0}^{k-1}\sqrt{\eta_{i+1}} \geq \sqrt{\eta_k}t_k - \sqrt{\eta_0}t_0 \geq \frac{1}{2}\sum_{i=0}^{k-1}\sqrt{\eta_{i+1}}.$$
    Therefore, 
    $$\frac{\sum_{i=0}^{k}\sqrt{\eta_i}}{\sqrt{\eta_k}} \geq t_k \geq \frac{\sum_{i=0}^{k}\sqrt{\eta_i}}{2\sqrt{\eta_k}}$$
    and we obtain the desired result. 
\end{proof}

From this result, we can conclude that,
\begin{align*}
    d(\lambda) - d(\lambda_{T}) \leq \frac{4GD_h(\lambda, \lambda_0)}{\left(\sum_{k=0}^{T-1}\sqrt{\eta_k}\right)^2}, \quad \forall \lambda \in \Lambda,
\end{align*}
which is the same as the previous case when $A\to +\infty$ (namely $\theta_0 = 1$).

% Note that when choosing $\eta_k = \eta, k=0,1,\cdots$, it can be shown that \cite{beck2009fast, nedelcu2014computational} 
% \[\frac{1}{k+1} \leq \theta_k \leq \frac{2}{k+1}\]
% Thus this proof also provides accelerated convergence rate. 

\section{Accelerated Bregman Augmented Lagrangian Method}\label{sec:acc-BALM}

Applying the acc-BPP algorithms to the dual problem associated with the linearly constrained convex programs would then lead to accelerated versions of BALM. We present in \algref{Accelerated BALM}, an accelerated BALM algorithm based on~\algref{Accelerated BPP3}, denoted as acc-BALM. As an immediate result, the dual sequence from acc-BALM converges in the rate of $\cO\big(1/(\sum_{k=0}^{T-1}\sqrt{\eta_k})^2\big)$, which improves over the $\cO(1/\sum_{k=0}^{T-1}\eta_k)$ of BALM. However, as discussed in the introduction, algorithms with an accelerated rate of dual convergence does not necessarily exhibit accelerated primal convergence. For example, the accelerated algorithm of ALM established in~\cite{he2010acceleration, ke2017accelerated, kang2013accelerated} with constant proximal parameters only ensures a $\cO(1/T^2)$ rate for the dual convergence, namely, $L(x^*,\lambda^*)-L(x_T,\lambda_T)\leq \cO(1/T^2)$,  whereas the primal sequence converges only in the rate of $\cO(1/T)$. 

Below we show that the proposed acc-BALM algorithm based on the previous acceleration scheme also ensures acceleration on the primal convergence.  From Section~\ref{sec:ergodic convergence}, we know that the key to prove primal convergence is to bound $L(\tilde{x}_T, \lambda) - L(x, \tilde{\lambda}_T)$, which further implies bounds for the primal objective  $|f(\tilde{x}_T)-f(x^*)|$ and feasibility violation $\|A\tilde{x}_T-b\|$. The next theorem establishes the primal convergence rate for \algref{Accelerated BALM}. 

\begin{algorithm}[ht!]
\KwIn{$\lambda_0 \in \Lambda, v_0=\lambda_0, G, \theta_0=1 , \{\eta_k\}_{k\geq0}$}
%\KwOut{Output}
\nl \For{$k\geq 0$}{
\nl $y_k = \theta_k v_k + (1-\theta_k)\lambda_k$\;
\nl $x_{k+1} \in \arg\min\left\{f(x) + \max_{\lambda\in \Lambda}\{\lambda^\top (Ax-b) - \frac{1}{\eta_{k}}D_h(\lambda, y_k)\}\right\}$\;
\nl $\lambda_{k+1} = \arg\max_{\lbd\in \Lambda}\{\lbd^\top (A_{k+1}-b) - \frac{1}{\eta_k}D_h(\lbd, y_k)\}$\;
\nl $v_{k+1} = \arg\max_{\lambda\in \Lambda}\left\{- GD_h(\lambda, \lambda_0) + \sum_{j=0}^k \frac{\eta_j}{\theta_j}\left(d(\lambda_{j+1}) + \frac{1}{\eta_j}(\nabla h(\lambda_{j+1}) - \nabla h(y_j))^\top (\lambda-\lambda_{j+1})\right)  \right\}$\;
\nl Update $\theta_{k+1}$ such that $\sum_{j=0}^{k+1}\frac{\eta_j}{\theta_j} = \frac{\eta_{k+1}}{\theta_{k+1}^2}$, satisfied by $\frac{\eta_k}{\theta_{k}^2} = \frac{\eta_{k+1}}{\theta_{k+1}^2} - \frac{\eta_{k+1}}{\theta_{k+1}}$.
% \nl Update $\theta_{k+1}$ such that $\frac{\eta_k}{\theta_{k}^2} = \frac{\eta_{k+1}}{\theta_{k+1}^2} - \frac{\eta_{k+1}}{\theta_{k+1}}$.
}
    \caption{{\bf  acc-BALM} \label{Accelerated BALM}}
\end{algorithm}

% \begin{lemma}\label{lma:primal1step} We have for any $x\in \mathcal{X},\lambda \in \Lambda$,
% \begin{align}
%     &L(x_{k+1}, \lambda) - L(x, \lambda_{k+1}) 
%     % =&f(x_{k+1}) - f(x) + \lambda^\top (Ax_{k+1}-b) - \lambda_{k+1}^\top (Ax-b) \nonumber\\
%     \leq \frac{1}{\eta_k}(\nabla h(\lambda_{k+1}) - \nabla h(y_k))^\top (\lambda - \lambda_{k+1}).
% \end{align}
% \end{lemma}
% \begin{proof}
% The proof is similar to that of \cref{lma:1step}.
% \end{proof}

\begin{theorem}
Denote
$$\tilde{x}_{T} = \frac{\sum_{k=0}^{T-1}(\eta_k/\theta_k)x_{k+1}}{\sum_{k=0}^{T-1}\eta_k/\theta_k},\; \tilde{\lbd}_{T} = \frac{\sum_{k=0}^{T-1}(\eta_k/\theta_k)\lbd_{k+1}}{\sum_{k=0}^{T-1}\eta_k/\theta_k}.$$ Then for any $x\in \mathcal{X}, \lambda \in \Lambda$, we have
\begin{align}
    L(\tilde{x}_{T}, \lambda) - L(x, \tilde{\lbd}_{T}) \leq \frac{GD_h(\lbd, \lbd_0) + GD_h(\lbd^*, \lbd_0)\sum_{k=0}^{T-1}\theta_k}{S_{T-1}}, 
\end{align}
where $S_{T-1} = \sum_{k=0}^{T-1}\eta_k/\theta_k$.
\end{theorem}
\begin{proof}
Using Theorem~\ref{thm:accdual} and Lemma~\ref{lma:1step}, we have $\forall x\in \mathcal{X}, \lambda \in \Lambda$,
\begin{eqnarray*}
   &&L(\tilde{x}_{T}, \lambda) - L(x, \tilde{\lbd}_{T}) \nonumber\\
    &\leq & \frac{1}{S_{T-1}}{\sum}_{k=0}^{T-1}(\eta_k/\theta_k)\left(L(x_{k+1}, \lambda) - L(x, \lbd_{k+1})\right) \nonumber\\
    &\leq & \frac{1}{S_{T-1}}\Big({\sum}_{k=0}^{T-1}(\eta_k/\theta_k)
    \big(\frac{1}{\eta_k}(\nabla h(\lambda_{k+1}) - \nabla h(y_k))^\top (\lambda - \lambda_{k+1})\big) \Big)\nonumber\\
    &\leq & \frac{1}{S_{T-1}}\Big(GD_h(\lbd, \lbd_0) + {\sum}_{k=0}^{T-1}(\eta_k/\theta_k)(d(\lbd_{T}) - d(\lbd_{k+1})) \Big)\nonumber
%     &\leq & \frac{1}{S_{T-1}}\left(GD_h(\lbd, v_0) + {\sum}_{k=0}^T(\eta_k/\theta_k)(d(\lbd^*) - d(\lbd_{k+1})) \right)\nonumber\\
%     &\leq & \frac{1}{S_{T-1}}\left(GD_h(\lbd, v_0) + {\sum}_{k=0}^T(\eta_k/\theta_k)\frac{\theta_k^2GD_h(\lambda^*, \lambda_0)}{\eta_k} \right)\nonumber\\ 
%   & = & \frac{1}{S_{T-1}}\left(GD_h(\lbd, \lbd_0) + GD_h(\lbd^*, \lbd_0){\sum}m_{k=0}^T\theta_k \right)
\end{eqnarray*}
where the first inequality uses the fact that $L(x,\lambda)$ is convex in $x$ and linear in $\lambda$, the second inequality applies Lemma~\ref{lma:1step}, and the third inequality comes from Theorem~\ref{thm:accdual}.  From Corollary~\ref{cor:dualconvergence}, we have $d(\lambda_{T})-d(\lbd_{k+1})\leq d(\lbd^*) - d(\lbd_{k+1})\leq \frac{\theta_k^2GD_h(\lambda^*, \lambda_0)}{\eta_k},\forall k$. It then leads to the desired result.
% \begin{eqnarray*}
%   L(\tilde{x}_{T}, \lambda) - L(x, \tilde{\lbd}_{T})
%     % &\leq & \frac{1}{S_{T-1}}\left(GD_h(\lbd, v_0) + {\sum}_{k=0}^T(\eta_k/\theta_k)(d(\lbd^*) - d(\lbd_{k+1})) \right)\nonumber\\
%     \leq  \frac{1}{S_{T-1}}\left(GD_h(\lbd, \lbd_0) + {\sum}_{k=0}^{T-1}(\eta_k/\theta_k)\frac{\theta_k^2GD_h(\lambda^*, \lambda_0)}{\eta_k} \right),\nonumber\\ 
% %   & = & \frac{1}{S_{T-1}}\left(GD_h(\lbd, \lbd_0) + GD_h(\lbd^*, \lbd_0){\sum}_{k=0}^T\theta_k \right)
% \end{eqnarray*}
\end{proof}

The following result can be obtained following the same proof as Theorem \ref{thm:primalexact}.
\begin{corollary}\label{thm:accprimalexact}
% Let $\tilde{x}_{T+1} = \frac{1}{\sum_{k=0}^T\eta_k/\theta_k}\sum_{k=0}^T(\eta_k/\theta_k)x_{k+1}$, $S_T = \sum_{k=0}^T\eta_k/\theta_k$.
Define $\rho_*=2\|\lambda^*\|+1$.  \algref{Accelerated BALM} satisfies that 
\begin{enumerate}
  \item[(a)] For the equality constrained problem \eq{eq:linear-equality-problem}, 
  \begin{align*}\label{eq:accbound_equality_case}
\max\left\{ |f(\tilde{x}_{T}) - f(x^*)|, \|A\tilde{x}_{T}-b\|_{\star}\right\}
\leq \frac{\max\limits_{\lambda \in \mathcal{B}_{\rho_*}}GD_h(\lbd, \lbd_0)(1+\sum_{k=0}^{T-1}\theta_k)}{S_{T-1}},
\end{align*}
where $\mathcal{B}_{\rho} = \{\lambda\in\RR^m : \|\lambda\|\leq \rho\}$ and $\|\cdot\|_\star$ is the dual norm of $\|\cdot\|$.
  \item[(b)] For the inequality constrained problem \eq{eq:linear-inequality-problem}, 
  \begin{align*}%\label{eq:accbound_inequality_case}
\max\left\{ |f(\tilde{x}_{T}) - f(x^*)|, \|[A\tilde{x}_{T}-b]_+\|_{\star}\right\}
\leq \frac{\max\limits_{\lambda \in \mathcal{B}^+_{\rho_*}}GD_h(\lbd, \lbd_0)(1+\sum_{k=0}^{T-1}\theta_k)}{S_{T-1}},
\end{align*}
where $\mathcal{B}^+_{\rho} = \{\lambda\in\RR^m :\lambda \geq 0, \|\lambda\|\leq \rho\}$.
\end{enumerate}
\end{corollary}

\paragraph{Discussions} From Proposition~\ref{prop:thetabound}, it is clear that $S_{T-1} \geq \left(\frac{\sum_{k=0}^{T-1}\sqrt{\eta_k}}{2}\right)^2$, and $\sum_{k=0}^{T-1}\theta_k \leq \sum_{k=0}^{T-1} \frac{2\sqrt{\eta_k}}{\sum_{i=0}^k\sqrt{\eta_i}}$. We now discuss the consequences of special choices of $\{\eta_k\}_{k\geq 0}$. In particular, we consider choosing $\eta_k = \eta(k+1)^p, k=0,1,\cdots$, where $p\geq 0$. First observe that
\begin{align*}
    \sum_{k=0}^{T-1}\sqrt{\eta_k} \geq \sum_{k=0}^{T-1}\sqrt{\eta} k^{p/2} \geq \sqrt{\eta} \int_0^{T} x^{p/2}dx \geq \frac{\sqrt{\eta}T^{p/2 + 1}}{p/2 + 1},
\end{align*}
thus $S_{T-1} \geq \frac{\eta T^{p+2}}{(p+2)^2}$. In addition, we have
\begin{align*}
    \sum_{k=0}^{T-1}\theta_k \leq \sum_{k=0}^{T-1}\frac{(p+2)(k+1)^{p/2}}{(k+1)^{p/2 + 1}} \leq\sum_{k=0}^{T-1}\frac{p+2}{k+1} \leq (p+2)\ln(T).
\end{align*}
Therefore, when the proximal parameters are set to $\eta_k = \eta(k+1)^p$, the primal convergence rate of acc-BALM becomes $\cO \left(\frac{\ln T}{T^{p+2}}\right)$, whereas the primal convergence rate of BALM is $\cO\left(\frac{1}{T^{p+1}}\right)$. In particular, when the proximal parameters are fixed to a constant, namely $p=0$, acc-BALM improves the primal convergence from $\cO(\frac{1}{T})$ to $\cO\left(\frac{\ln(T)}{T^2}\right)$.  When the Bregman divergence is set to the Euclidean distance, and the constraints are linear equality constraints, our result recovers the existing result in \cite{nedelcu2014computational} as a special case. However, it is worth mentioning that our acceleration scheme and proof techniques are fairly general, whereas the accelerated algorithm and the convergence proof in \cite{nedelcu2014computational} heavily rely on the structure of linear equality constraints and only apply to classical ALM. We believe our result gives the first primal convergence analysis of accelerated ALM for problems with inequality constraints. 

\section{Numerical Experiments}\label{sec:experiments}
In this section, we test the numerical performance of the proposed algorithms, particularly, acc-BPP in~\algref{Accelerated BPP2} and acc-BALM in~\algref{Accelerated BALM}, and compare to their non-accelerated counterparts.

 We first consider two different convex problems:
\begin{equation}\label{eq:problem1}
  \min_{x\in \Delta_n}\;f_1(x):=\max\{c_1^\top x,\cdots, c_m^\top x\}, 
\end{equation}
\begin{equation}\label {eq:problem2}
   \min_{x\in \Delta_n} f_2(x):={\sum}_{j=1}^m\exp(a_j^\top x), 
\end{equation}
where $\Delta_n:=\{x\in\mathbb{R}^n: {\sum}_{i=1}^nx_i = 1, x\geq 0\}$ is the simplex set. These two examples represent nonsmooth and smooth convex objectives, respectively. For both problems, we consider $m=15, n=20$. For problem~\eq{eq:problem1}, each $c_j$ is randomly generated from $U[-1,1]^n$. For problem~\eq{eq:problem2}, each $a_j$ is also uniformly generated from $U[-1,1]^n$.

We run BPP and acc-BPP to solve the above two problems. We choose the Bregman function $h(x) = \sum_{i=1}^nx_i\ln(x_i):\Delta_n\to\mathbb{R}$, and the Bregman divergence is $D_h(x,y) = \sum_{i=1}^n(x_i \ln(x_i/y_i)-x_i+y_i)$, which is also known as the \textit{generalized KL-divergence}.  The Bregman operators (i.e.,  the optimal solutions to the proximal minimization steps) are obtained using ECOS conic programming solver version 2.0.7~\cite{domahidi2013ecos}. Even though $D_h(x,y)$ here does not strictly satisfy the triangle-scaling property \eq{eq:trianglescaling}, we simply setting $G=1$ in the experiment and consider two different choices of proximal parameters: $\eta_k=1$ and $\eta_k=k+1$. Results are summarized in Figure~\ref{fig1}, which  indicates that  acc-BPP achieves  faster convergences than BPP under both settings.

\begin{figure}[!ht]
\centering
\begin{minipage}{.5\textwidth}
 \centering
 \includegraphics[width=1.0\linewidth]{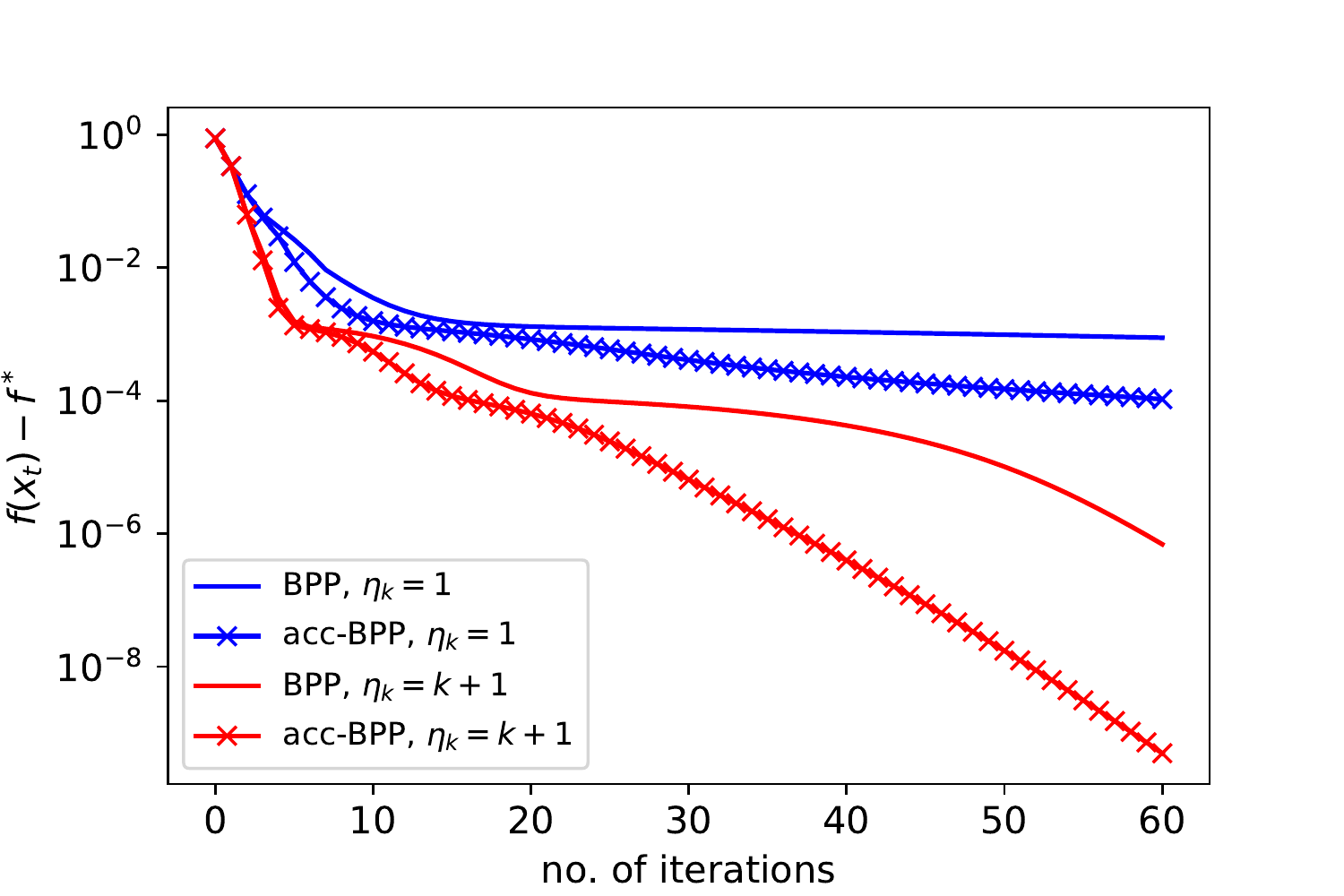}
 \captionsetup{font={small}, margin={0cm,0cm}}
 \caption*{(a) problem \eq{eq:problem1}}
% \label{fig:test1}
\end{minipage}%
%\vspace{0.0pt}
\begin{minipage}{.5\textwidth}
 \centering
 \includegraphics[width=1.0\linewidth]{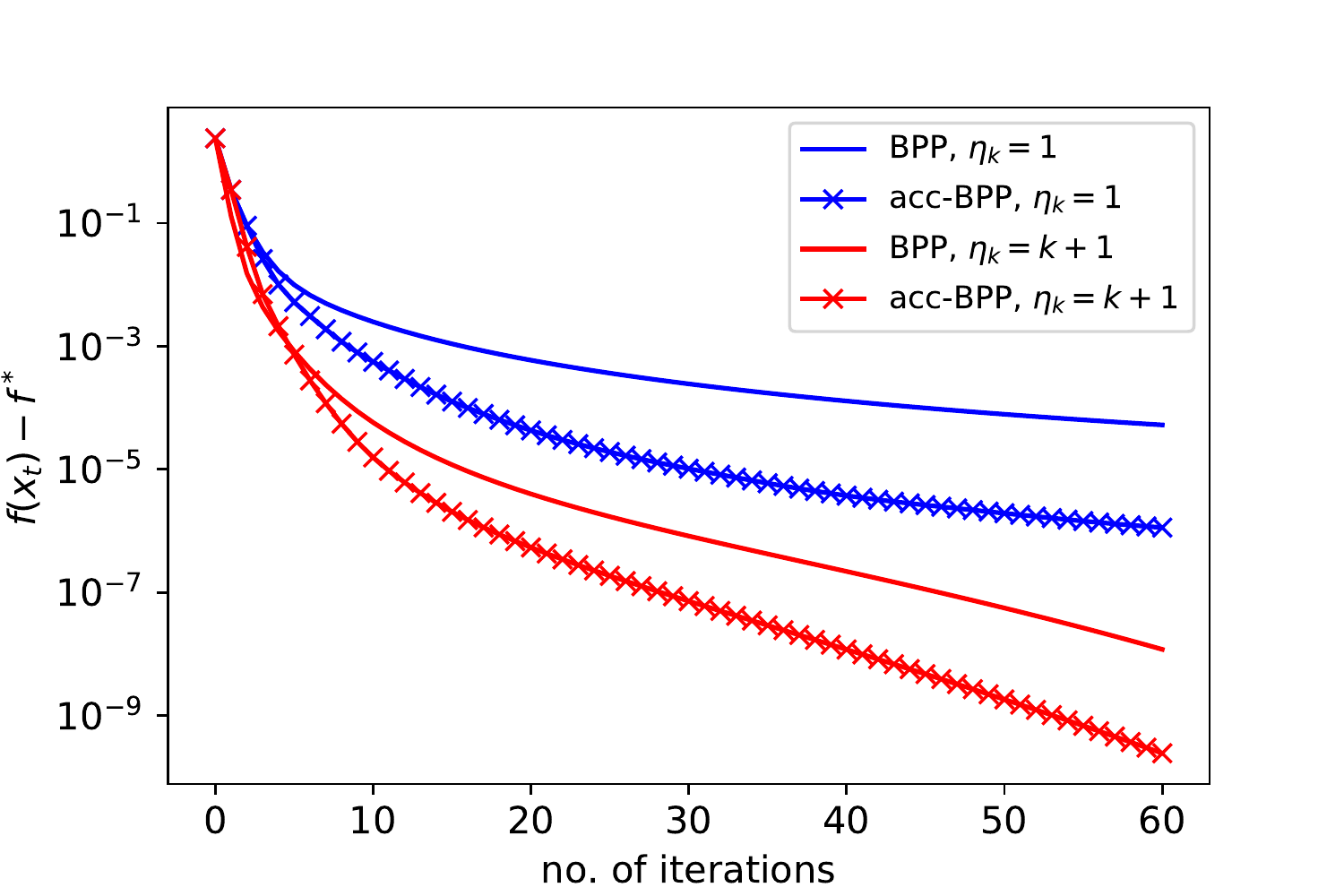}
 \captionsetup{font={small}, margin={0cm,0cm}}
 \caption*{(b) problem \eq{eq:problem2}}
%\label{fig:test2}
\end{minipage}
\captionsetup{font={small}}
\caption{Comparison of BPP and acc-BPP on convex  problems \eq{eq:problem1} and \eq{eq:problem2}} 
\label{fig1}
\end{figure}

\begin{figure}[ht!]
\centering
\begin{minipage}{.5\textwidth}
 \centering
 \includegraphics[width=1.0\linewidth]{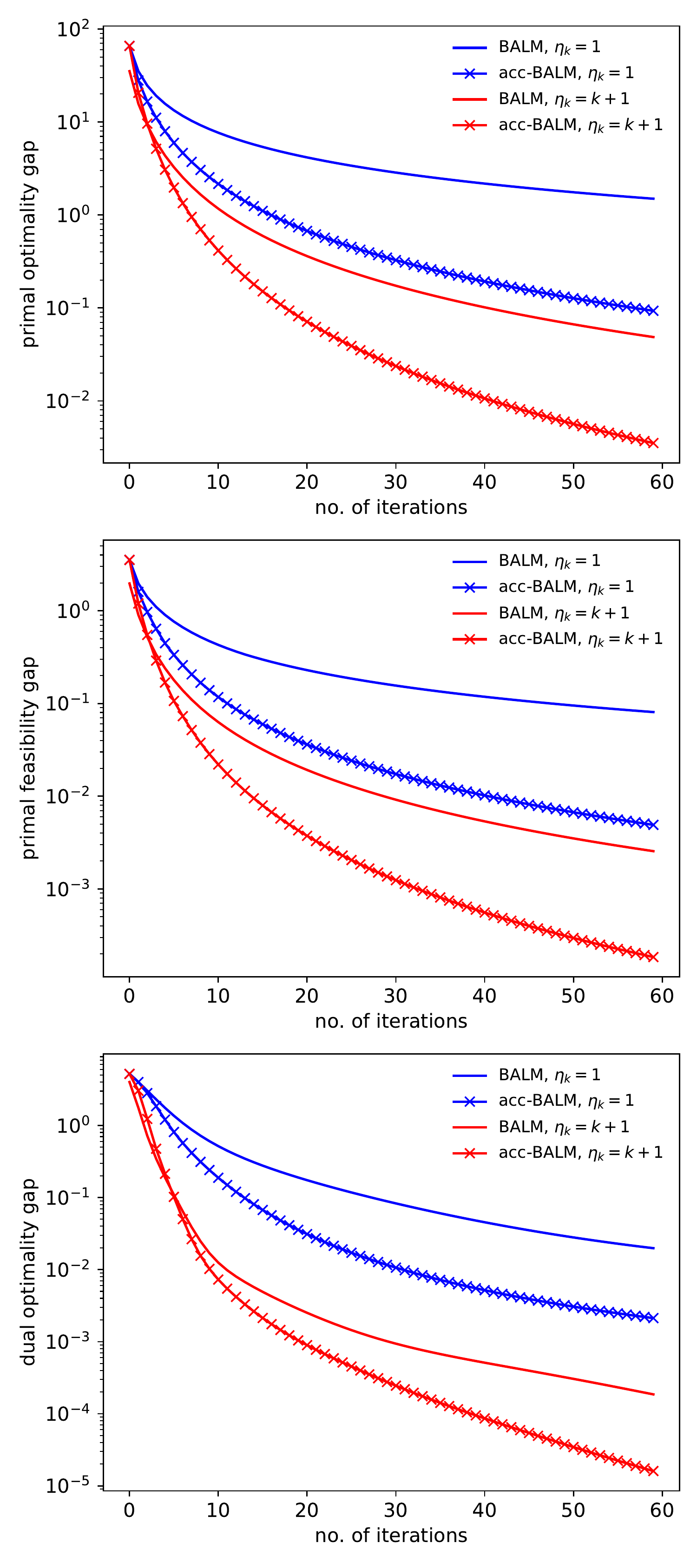}
 \captionsetup{font={small}}
 \caption*{(a) problem \eq{eq:problem3}}
% \label{fig:test1}
\end{minipage}%
\begin{minipage}{.5\textwidth}
 \centering
 \includegraphics[width=1.0\linewidth]{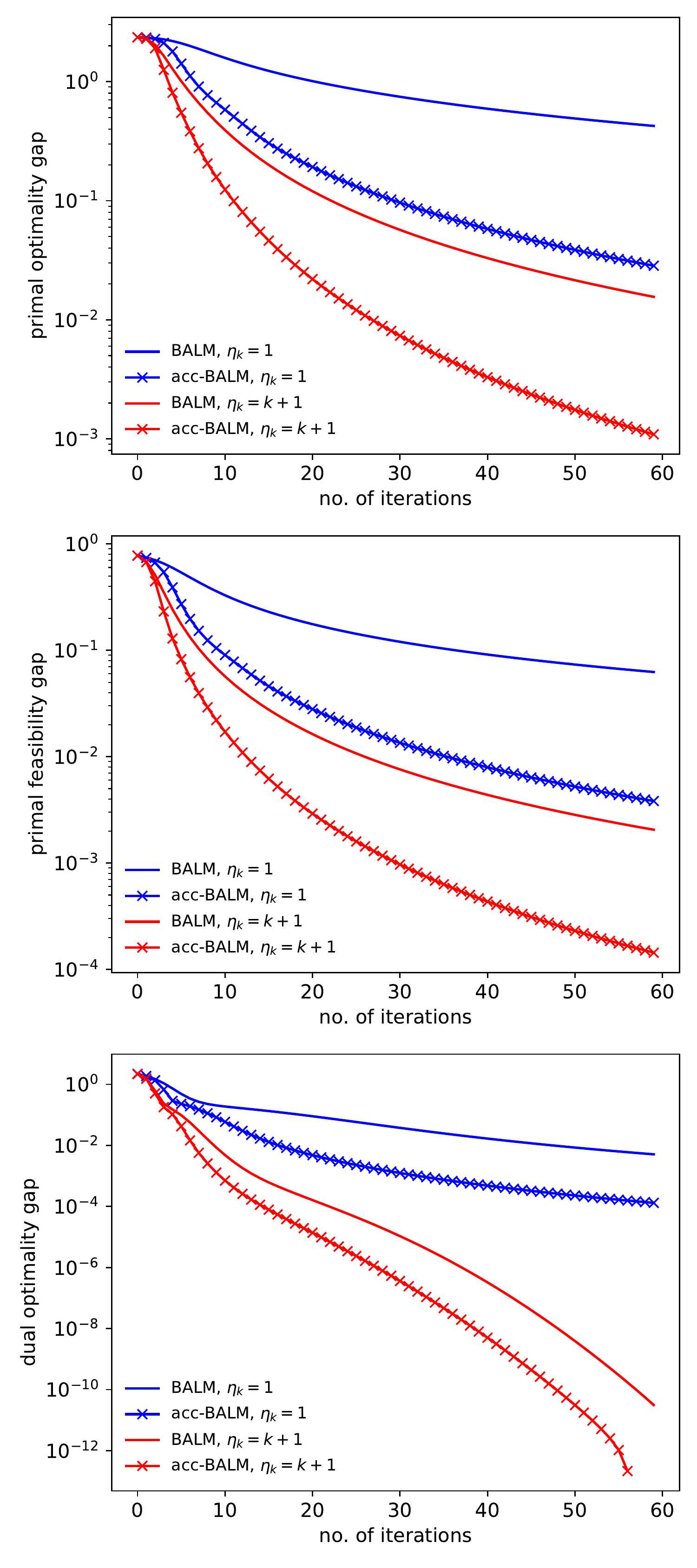}
 \captionsetup{font={small}}
 \caption*{(b) problem \eq{eq:problem4}}
% \label{fig:test2}
\end{minipage}
\captionsetup{font={small}}
\caption{Comparison of BALM and acc-BALM on convex problems \eq{eq:problem3} and \eq{eq:problem4}\\[1cm]}
\label{fig2}
\end{figure}

Next, we consider another two convex minimization problems with linear inequality constraints for evaluating the performance of BALM and acc-BALM:
\begin{equation}\label{eq:problem3}
\min_x\{c^\top x: Ax\leq b\}, 
\end{equation}
\begin{equation}\label{eq:problem4}
  \min_x\left\{ \frac{1}{2}x^\top W x: Ax\leq b \right\}
\end{equation}
where $A\in \mathbb{R}^{m\times n}$. For both problems, we set $m=150, n=30$. For problem~\eq{eq:problem3}, we first generate an instance of \textit{Markov decision problem}~\cite{puterman2014markov}, where each entry of transition probability is randomly generated from $U[0,1]$ with normalization, and rewards are also uniformly generated from $U[0,1]$. We then consider the linear program formulation associated with this finite MDP. For problem~\eq{eq:problem4}, we set $W = \omega^\top \omega$, where $\omega \sim U[0,2]^n$, and $A \sim U[0,1]^{m\times n}, b \sim U[-1,1]^m$. We set the Bregman divergence to be the generalized KL-divergence for both BALM and acc-BALM. The optimal solutions to the subproblems are obtained using ECOS solver version 2.0.7~\cite{domahidi2013ecos}. Again, we simply setting $G=1$ in the experiment and consider two different choices of proximal parameters: $\eta_k=1$ and $\eta_k=k+1$. Results are summarized in Figure~\ref{fig2}, which clearly indicates that  acc-BALM achieves faster convergences than BALM under both settings.

\section{Conclusions}
\label{sec:conclusions}
We have established the first non-asymptotic primal convergence rate for BALM, which generalizes the classical ALM method. A generic accelerated scheme of the Bregman proximal point method is proposed, which is further used to construct the first accelerated BALM with both improved dual and  primal convergence rates. 
% The accelerated primal convergence rate also extends the classical ALM, and also covers broder problems such as inequality constrained problems. 
Numerical experiments demonstrate that these accelerated algorithms achieve superior performance in practice. 
For future work, it remains interesting to explore the total iteration complexity of these accelerated algorithms and when subproblems are solved inexactly through some (stochastic) first-order subroutines. 
% even when the Bregman divergence is not well-behaved.

\vspace{1cm}
\appendix
\noindent{\Large \textbf{Appendix}}
\section{Existing accelerated ALM methods}\label{appendix}
Below, we summarize several existing accelerated ALM schemes for solving linear equality problem \eq{eq:linear-equality-problem}. 
% Given feasible $y_0$ and $t_0=1$, \\
\begin{enumerate}
\item Applying G\"uler's 1st accelerated proximal point method to the Lagrangian  dual \cite{guler1992new, kim2019accelerated, beck2009fast, tseng2008accelerated, lin2018catalyst}:
\begin{align}
\left\{ \begin{array}{l}
x_{k+1} = \arg\min_{x\in \mathcal{X}}\{f(x) + y_k^\top (Ax-b) + \frac{\eta}{2}\|Ax-b\|^2\}\\
\lambda_{k+1} = y_k + \eta (Ax_{k+1} - b)\\
t_{k+1} = \frac{1+\sqrt{1+4t_k^2}}{2}\\
y_{k+1} = \lambda_{k+1} + \frac{t_k - 1}{t_{k+1}}(\lambda_{k+1}-\lambda_k)
\end{array} \right. \label{guler1st}
\end{align}

\item Applying G\"uler's 2nd accelerated proximal point method to the Lagrangian dual \cite{guler1992new, kim2019accelerated, he2010acceleration, ke2017accelerated, kang2013accelerated}:
\begin{align}
    \left\{ \begin{array}{l}
    x_{k+1} = \arg\min_{x\in \mathcal{X}}\{f(x) + y_k^\top (Ax-b) + \frac{\eta}{2}\|Ax-b\|^2\}\\
    \lambda_{k+1} = y_k + \eta (Ax_{k+1} - b)\\
    t_{k+1} = \frac{1+\sqrt{1+4t_k^2}}{2}\\
    y_{k+1} = \lambda_{k+1} + \frac{t_k - 1}{t_{k+1}}(\lambda_{k+1}-\lambda_k) + \frac{t_k}{t_{k+1}}(\lambda_{k+1}-y_k)
\end{array} \right. \label{guler2nd}
\end{align}

\item Applying Nesterov's accelerated dual average method to the augmented Lagrangian dual problem \cite{nesterov2005smooth, devolder2014first, nedelcu2014computational}:
\begin{align}
    \left\{ \begin{array}{l}
x_{k+1} = \arg\min_{x\in \mathcal{X}}\{f(x) + y_k^\top (Ax-b) + \frac{\eta}{2}\|Ax-b\|^2\}\\
\lambda_{k+1} = y_{k} + \eta (Ax_{k+1}-b)\\
t_{k+1} = \frac{1+\sqrt{1+4t_k^2}}{2}\\
y_{k+1} = \left(1-\frac{1}{t_{k+1}}\right)\lambda_{k+1} + \frac{1}{t_{k+1}}\left(\lambda_0 + \eta \sum_{j=0}^k t_j (Ax_{j+1}-b)\right)
\end{array} \right. \label{dualaverage}
\end{align}
\end{enumerate}

\section{Equivalence of \eq{guler1st} and \eq{dualaverage}}
The last step of \eq{dualaverage} can be rewritten as the following (with $v_0 = y_0$)
\begin{align*}
\left\{ \begin{array}{l}
v_{k+1} = v_k + t_k(\lambda_{k+1} - y_k)\\
y_{k+1} = \left(1-\frac{1}{t_{k+1}}\right)\lambda_{k+1} + \frac{1}{t_{k+1}}v_{k+1}
\end{array} \right.
\end{align*}
Eliminate $v_k$ and using the fact that $v_k = t_ky_k - (t_k - 1)\lambda_k$, we have $v_{k+1} = v_k + t_k(\lambda_{k+1} - y_k)$. This implies that
    $$t_{k+1}y_{k+1} - (t_{k+1} - 1)\lambda_{k+1} = t_ky_k - (t_k - 1)\lambda_k + t_k (\lambda_{k+1} - y_k).$$
    Equivalent, $y_{k+1} = \lambda_{k+1} + \frac{t_k-1}{t_{k+1}}(\lambda_{k+1} - \lambda_k)$, which recovers last step of \eq{guler1st}. 
% Since it is proved in \cite{nedelcu2014computational} that \cref{dualaverage} achieves $\cO(1/T^2)$ convergence rate for primal objective and constraints, i.e. both $|f(\tilde{x}_T) - f(x^*)|$ and $\|A\tilde{x}_T-b\|$ are bounded by terms of order $\cO(1/T^2)$, it must also be true for \cref{guler1st}.

\section{Primal convergence rate of \eq{guler2nd}} Existing proofs only indicate an accelerated rate of dual convergence, yet it remains unclear whether this method could guarantee an improved rate of primal acceleration at the same time. Here we provide a simple counter-example showing that the algorithm described in \eq{guler2nd} could fail to improve the primal convergence. The example we consider is a simple linear program:
\begin{align*}
\min_{x\in\mathbb{R}^n}\{c^\top x: Ax=b\},
\end{align*}
where we assume $A\in \mathbb{R}^{m\times n}$, $m > n$, and $\text{rank}(A) = n$. The corresponding dual problem is
\begin{align*}
\max_{\lambda\in\mathbb{R}^m}\{-\lambda^\top b: A^\top \lambda + c = 0 \}.
\end{align*}
Assume that the problem is feasible, i.e., $b \in \text{range}(A)$. This implies that there is only one feasible solution, which we shall call $x^*$. Clearly, $x^* = (A^\top A)^{-1}A^\top b$, and $c^\top x^* = c^\top (A^\top A)^{-1}A^\top b$. Now \eq{guler2nd} for this specific problem can be written as (with $t_0 = 1, y_0 = \mathbf{0}$)

\begin{align*}
\left\{ \begin{array}{l}
x_{k+1} = (A^\top A)^{-1}(A^\top b - \frac{1}{\eta}(A^\top y_k + c))\\
\lambda_{k+1} = y_k + \eta (Ax_{k+1} - b)\\
t_{k+1} = \frac{1+\sqrt{1+4t_k^2}}{2}\\
y_{k+1} = \lambda_{k+1} + \frac{t_k - 1}{t_{k+1}}(\lambda_{k+1}-\lambda_k) + \frac{t_k}{t_{k+1}}(\lambda_{k+1}-y_k)
\end{array} \right..
\end{align*}

\noindent By substituting $x_{k+1}$, we get
\begin{align*}
\lambda_{k+1} = (I - A(A^\top A)^{-1}A^\top)y_k - A(A^\top A)^{-1}c,
\end{align*}
from this expression we can see that the dual variable $\lambda_1$ already recovers the optimal solution, since $A^\top \lambda_1 + c = 0$, and $-b^\top \lambda_1 = c^\top (A^\top A)^{-1}A^\top b$. We further obtain the update for $y_{k+1}$:
\begin{align*}
    y_{k+1} = (I - A(A^\top A)^{-1}A^\top)&\left(y_k + \frac{t_k-1}{t_{k+1}}( y_k - y_{k-1})\right) - A(A^\top A)^{-1}c \\
    &- \frac{t_k}{t_{k+1}}\left(A(A^\top A)^{-1}A^\top y_k + A(A^\top A)^{-1}c\right).
\end{align*}
To simplify notations, let $z_k = A^\top y_k + c$, we therefore have
\begin{align*}
    z_{k+1} &= -\frac{t_k}{t_{k+1}}z_k = (-1)^{k+1}\frac{t_0}{t_{k+1}}c.
\end{align*}
We can then obtain the  following update for $x_{k+1}$,
\begin{align*}
    x_{k+1} &= (A^\top A)^{-1}(A^\top b - \frac{1}{\eta}(A^\top y_k + c))\\
          &= x^* - \frac{1}{\eta}(A^\top A)^{-1}z_k\\
          &= x^* - \frac{(-1)^k(A^\top A)^{-1}c}{\eta}\frac{t_0}{t_k}.
\end{align*}
This leads to the primal optimality and feasibility gap, 
\begin{align*}
|c^\top(x^*-x_T)| &= \frac{c^\top (A^\top A)^{-1}c}{\eta}\frac{t_0}{t_{T-1}},\\
\|Ax_T-b\| &= \frac{\|A(A^\top A)^{-1}c\|}{\eta}\frac{t_0}{t_{T-1}}.
\end{align*}
Since $\frac{T+1}{2} \leq t_T \leq T+1$, for this problem, \eq{guler2nd} only achieves $\cO(1/T)$ primal convergence rate.

 \section*{Acknowledgments}
 We would like to acknowledge Bo Dai and Donghwan Lee for fruitful discussions, and we thank Lin Xiao and Jonathan Eckstein for insightful comments. 

\bibliographystyle{unsrtnat}  
\bibliography{references.bib}

\begin{thebibliography}{47}
\providecommand{\natexlab}[1]{#1}
\providecommand{\url}[1]{\texttt{#1}}
\expandafter\ifx\csname urlstyle\endcsname\relax
  \providecommand{\doi}[1]{doi: #1}\else
  \providecommand{\doi}{doi: \begingroup \urlstyle{rm}\Url}\fi

\bibitem[Hestenes(1969)]{hestenes1969multiplier}
Magnus~R Hestenes.
\newblock Multiplier and gradient methods.
\newblock \emph{Journal of optimization theory and applications}, 4\penalty0
  (5):\penalty0 303--320, 1969.

\bibitem[Powell(1969)]{powell1969method}
Michael~JD Powell.
\newblock A method for nonlinear constraints in minimization problems.
\newblock \emph{Optimization}, pages 283--298, 1969.

\bibitem[Bertsekas(2014)]{bertsekas2014constrained}
Dimitri~P. Bertsekas.
\newblock \emph{Constrained optimization and Lagrange multiplier methods}.
\newblock Academic press, 2014.

\bibitem[Rockafellar(1976)]{rockafellar1976augmented}
R.~Tyrrell Rockafellar.
\newblock Augmented {L}agrangian and applications of the proximal point
  algorithm in convex programming.
\newblock \emph{Mathematics of operations research}, 1\penalty0 (2):\penalty0
  97--116, 1976.

\bibitem[Lan and Monteiro(2016)]{lan2016iteration}
Guanghui Lan and Renato~DC Monteiro.
\newblock Iteration-complexity of first-order augmented {L}agrangian methods
  for convex programming.
\newblock \emph{Mathematical Programming}, 155\penalty0 (1-2):\penalty0
  511--547, 2016.

\bibitem[Xu(2019)]{xu2017iteration}
Yangyang Xu.
\newblock Iteration complexity of inexact augmented {L}agrangian methods for
  constrained convex programming.
\newblock \emph{Mathematical Programming}, 2019.

\bibitem[Chen et~al.(2017)Chen, Guo, Lu, and Ye]{chen2017augmented}
Xiaojun Chen, Lei Guo, Zhaosong Lu, and Jane~J Ye.
\newblock An augmented {L}agrangian method for non-lipschitz nonconvex
  programming.
\newblock \emph{SIAM Journal on Numerical Analysis}, 55\penalty0 (1):\penalty0
  168--193, 2017.

\bibitem[Sahin et~al.(2019)Sahin, Alacaoglu, Latorre, and
  Cevher]{sahin2019inexact}
Mehmet~Fatih Sahin, Ahmet Alacaoglu, Fabian Latorre, and Volkan Cevher.
\newblock An inexact augmented {L}agrangian framework for nonconvex
  optimization with nonlinear constraints.
\newblock In \emph{Advances in Neural Information Processing Systems}, pages
  13943--13955, 2019.

\bibitem[Ouyang et~al.(2013)Ouyang, He, Tran, and Gray]{ouyang2013stochastic}
Hua Ouyang, Niao He, Long Tran, and Alexander Gray.
\newblock Stochastic alternating direction method of multipliers.
\newblock In \emph{International Conference on Machine Learning}, pages 80--88,
  2013.

\bibitem[Xu(2017)]{xu2017accelerated}
Yangyang Xu.
\newblock Accelerated first-order primal-dual proximal methods for linearly
  constrained composite convex programming.
\newblock \emph{SIAM Journal on Optimization}, 27\penalty0 (3):\penalty0
  1459--1484, 2017.

\bibitem[Li et~al.(2018)Li, Sun, and Toh]{li2018highly}
Xudong Li, Defeng Sun, and Kim-Chuan Toh.
\newblock A highly efficient semismooth newton augmented {L}agrangian method
  for solving lasso problems.
\newblock \emph{SIAM Journal on Optimization}, 28\penalty0 (1):\penalty0
  433--458, 2018.

\bibitem[Lu and Zhang(2012)]{lu2012augmented}
Zhaosong Lu and Yong Zhang.
\newblock An augmented {L}agrangian approach for sparse principal component
  analysis.
\newblock \emph{Mathematical Programming}, 135\penalty0 (1-2):\penalty0
  149--193, 2012.

\bibitem[Yang et~al.(2015)Yang, Sun, and Toh]{yang2015sdpnal}
Liuqin Yang, Defeng Sun, and Kim-Chuan Toh.
\newblock {SDPNAL}: a majorized semismooth newton-cg augmented {L}agrangian
  method for semidefinite programming with nonnegative constraints.
\newblock \emph{Mathematical Programming Computation}, 7\penalty0 (3):\penalty0
  331--366, 2015.

\bibitem[G{\"u}ler(1992)]{guler1992new}
Osman G{\"u}ler.
\newblock New proximal point algorithms for convex minimization.
\newblock \emph{SIAM Journal on Optimization}, 2\penalty0 (4):\penalty0
  649--664, 1992.

\bibitem[Kim(2019)]{kim2019accelerated}
Donghwan Kim.
\newblock Accelerated proximal point method for maximally monotone operators.
\newblock \emph{arXiv preprint arXiv:1905.05149}, 2019.

\bibitem[Beck and Teboulle(2009)]{beck2009fast}
Amir Beck and Marc Teboulle.
\newblock A fast iterative shrinkage-thresholding algorithm for linear inverse
  problems.
\newblock \emph{SIAM journal on imaging sciences}, 2\penalty0 (1):\penalty0
  183--202, 2009.

\bibitem[Tseng(2010)]{tseng2010approximation}
Paul Tseng.
\newblock Approximation accuracy, gradient methods, and error bound for
  structured convex optimization.
\newblock \emph{Mathematical Programming}, 125\penalty0 (2):\penalty0 263--295,
  2010.

\bibitem[Lin et~al.(2017)Lin, Mairal, and Harchaoui]{lin2018catalyst}
Hongzhou Lin, Julien Mairal, and Zaid Harchaoui.
\newblock Catalyst acceleration for first-order convex optimization: from
  theory to practice.
\newblock \emph{The Journal of Machine Learning Research}, 18\penalty0
  (1):\penalty0 7854--7907, 2017.

\bibitem[Nedelcu et~al.(2014)Nedelcu, Necoara, and
  Tran-Dinh]{nedelcu2014computational}
Valentin Nedelcu, Ion Necoara, and Quoc Tran-Dinh.
\newblock Computational complexity of inexact gradient augmented {L}agrangian
  methods: application to constrained {MPC}.
\newblock \emph{SIAM Journal on Control and Optimization}, 52\penalty0
  (5):\penalty0 3109--3134, 2014.

\bibitem[He and Yuan(2010)]{he2010acceleration}
Bingsheng He and Xiaoming Yuan.
\newblock On the acceleration of augmented lagrangian method for linearly
  constrained optimization.
\newblock \emph{Optimization online}, 3, 2010.

\bibitem[Ke and Ma(2017)]{ke2017accelerated}
Yi-fen Ke and Chang-feng Ma.
\newblock An accelerated augmented {L}agrangian method for linearly constrained
  convex programming with the rate of convergence o (1/k 2).
\newblock \emph{Applied Mathematics-A Journal of Chinese Universities},
  32\penalty0 (1):\penalty0 117--126, 2017.

\bibitem[Kang et~al.(2013)Kang, Yun, Woo, and Kang]{kang2013accelerated}
Myeongmin Kang, Sangwoon Yun, Hyenkyun Woo, and Myungjoo Kang.
\newblock Accelerated bregman method for linearly constrained $\ell_1 - \ell_2
  $ minimization.
\newblock \emph{Journal of Scientific Computing}, 56\penalty0 (3):\penalty0
  515--534, 2013.

\bibitem[Devolder et~al.(2014)Devolder, Glineur, and
  Nesterov]{devolder2014first}
Olivier Devolder, Fran{\c{c}}ois Glineur, and Yurii Nesterov.
\newblock First-order methods of smooth convex optimization with inexact
  oracle.
\newblock \emph{Mathematical Programming}, 146\penalty0 (1-2):\penalty0 37--75,
  2014.

\bibitem[Nesterov(2005)]{nesterov2005smooth}
Yu~Nesterov.
\newblock Smooth minimization of non-smooth functions.
\newblock \emph{Mathematical programming}, 103\penalty0 (1):\penalty0 127--152,
  2005.

\bibitem[Censor and Zenios(1992)]{censor1992proximal}
Yair Censor and Stavros~Andrea Zenios.
\newblock Proximal minimization algorithm with {D}-functions.
\newblock \emph{Journal of Optimization Theory and Applications}, 73\penalty0
  (3):\penalty0 451--464, 1992.

\bibitem[Teboulle(1992)]{teboulle1992entropic}
Marc Teboulle.
\newblock Entropic proximal mappings with applications to nonlinear
  programming.
\newblock \emph{Mathematics of Operations Research}, 17\penalty0 (3):\penalty0
  670--690, 1992.

\bibitem[Teboulle(2018)]{teboulle2018simplified}
Marc Teboulle.
\newblock A simplified view of first order methods for optimization.
\newblock \emph{Mathematical Programming}, 170\penalty0 (1):\penalty0 67--96,
  2018.

\bibitem[Yuan et~al.(2017)Yuan, Yin, Bai, Feng, and Tai]{yuan2017bregman}
Jing Yuan, Ke~Yin, Yi-Guang Bai, Xiang-Chu Feng, and Xue-Cheng Tai.
\newblock Bregman-proximal augmented {L}agrangian approach to multiphase image
  segmentation.
\newblock In \emph{International Conference on Scale Space and Variational
  Methods in Computer Vision}, pages 524--534. Springer, 2017.

\bibitem[Eckstein(1993)]{eckstein1993nonlinear}
Jonathan Eckstein.
\newblock Nonlinear proximal point algorithms using bregman functions, with
  applications to convex programming.
\newblock \emph{Mathematics of Operations Research}, 18\penalty0 (1):\penalty0
  202--226, 1993.

\bibitem[Chen and Teboulle(1993)]{chen1993convergence}
Gong Chen and Marc Teboulle.
\newblock Convergence analysis of a proximal-like minimization algorithm using
  bregman functions.
\newblock \emph{SIAM Journal on Optimization}, 3\penalty0 (3):\penalty0
  538--543, 1993.

\bibitem[Auslender and Teboulle(2006)]{auslender2006interior}
Alfred Auslender and Marc Teboulle.
\newblock Interior gradient and proximal methods for convex and conic
  optimization.
\newblock \emph{SIAM Journal on Optimization}, 16\penalty0 (3):\penalty0
  697--725, 2006.

\bibitem[Hanzely et~al.(2018)Hanzely, Richtarik, and
  Xiao]{hanzely2018accelerated}
Filip Hanzely, Peter Richtarik, and Lin Xiao.
\newblock Accelerated bregman proximal gradient methods for relatively smooth
  convex optimization.
\newblock \emph{arXiv preprint arXiv:1808.03045}, 2018.

\bibitem[Wang and Banerjee(2014)]{wang2014bregman}
Huahua Wang and Arindam Banerjee.
\newblock Bregman alternating direction method of multipliers.
\newblock In \emph{Advances in Neural Information Processing Systems}, pages
  2816--2824, 2014.

\bibitem[Zhao et~al.(2015)Zhao, Yang, Zhang, and Li]{zhao2015adaptive}
Peilin Zhao, Jinwei Yang, Tong Zhang, and Ping Li.
\newblock Adaptive stochastic alternating direction method of multipliers.
\newblock In \emph{International Conference on Machine Learning}, pages 69--77,
  2015.

\bibitem[Tseng(2008)]{tseng2008accelerated}
Paul Tseng.
\newblock On accelerated proximal gradient methods for convex-concave
  optimization.
\newblock \emph{submitted to SIAM Journal on Optimization}, 2:\penalty0 3,
  2008.

\bibitem[Gutman and Pe{\~n}a(2018)]{gutman2018unified}
David~H. Gutman and Javier~F. Pe{\~n}a.
\newblock A unified framework for bregman proximal methods: subgradient,
  gradient, and accelerated gradient schemes.
\newblock \emph{arXiv preprint arXiv:1812.10198}, 2018.

\bibitem[Nesterov(1988)]{nesterov1988approach}
Yurii Nesterov.
\newblock On an approach to the construction of optimal methods of minimization
  of smooth convex functions.
\newblock \emph{Ekonomika i Mateaticheskie Metody}, 24\penalty0 (3):\penalty0
  509--517, 1988.

\bibitem[Rockafellar(2015)]{rockafellar2015convex}
Ralph~Tyrell Rockafellar.
\newblock \emph{Convex analysis}.
\newblock Princeton university press, 2015.

\bibitem[Ekeland and Temam(1999)]{ekeland1999convex}
Ivar Ekeland and Roger Temam.
\newblock \emph{Convex analysis and variational problems}, volume~28.
\newblock SIAM, 1999.

\bibitem[Tseng and Bertsekas(1993)]{tseng1993convergence}
Paul Tseng and Dimitri~P. Bertsekas.
\newblock On the convergence of the exponential multiplier method for convex
  programming.
\newblock \emph{Mathematical Programming}, 60\penalty0 (1-3):\penalty0 1--19,
  1993.

\bibitem[Iusem(1999)]{iusem1999augmented}
Alfredo~N. Iusem.
\newblock Augmented {L}agrangian methods and proximal point methods for convex
  optimization.
\newblock \emph{Investigaci{\'o}n Operativa}, 8\penalty0 (11-49):\penalty0 7,
  1999.

\bibitem[He and Yuan(2012)]{he2012accelerated}
Bingsheng He and Xiaoming Yuan.
\newblock An accelerated inexact proximal point algorithm for convex
  minimization.
\newblock \emph{Journal of Optimization Theory and Applications}, 154\penalty0
  (2):\penalty0 536--548, 2012.

\bibitem[Salzo and Villa(2012)]{salzo2012inexact}
Saverio Salzo and Silvia Villa.
\newblock Inexact and accelerated proximal point algorithms.
\newblock \emph{Journal of Convex analysis}, 19\penalty0 (4):\penalty0
  1167--1192, 2012.

\bibitem[Hamdi and Mukheimer(2011)]{hamdi2011convergence}
Abdelouahed Hamdi and Aiman~A. Mukheimer.
\newblock Convergence of an augmented {L}agrangian algorithm for solving
  minimization problems.
\newblock \emph{International Journal of Optimization: Theory, Methods and
  Applications}, 1\penalty0 (4):\penalty0 381--394, 2011.

\bibitem[Hamdi et~al.(2012)Hamdi, Noor, and Mukheimer]{hamdi2012convergence}
Abdelouahed Hamdi, Muhammad~Aslam Noor, and AA~Mukheimer.
\newblock Convergence of a proximal point algorithm for solving minimization
  problems.
\newblock \emph{Journal of Applied Mathematics}, 2012, 2012.

\bibitem[Domahidi et~al.(2013)Domahidi, Chu, and Boyd]{domahidi2013ecos}
Alexander Domahidi, Eric Chu, and Stephen Boyd.
\newblock Ecos: An socp solver for embedded systems.
\newblock In \emph{2013 European Control Conference (ECC)}, pages 3071--3076.
  IEEE, 2013.

\bibitem[Puterman(2014)]{puterman2014markov}
Martin~L. Puterman.
\newblock \emph{Markov Decision Processes.: Discrete Stochastic Dynamic
  Programming}.
\newblock John Wiley \& Sons, 2014.

\end{thebibliography}

\end{document}